\documentclass[aop,MSNbibl,nameyear,dvips]{arximspdf}

\doi{10.1214/13-AOP863} %kopijuoti is PTS
\volume{41}
\issue{6}
\pubyear{2013}
\firstpage{4050}
\lastpage{4079}

\makeatletter
\newcommand{\rright}{\right}
\newcommand{\lleft}{\left}
\newcommand{\rrvert}{\vert}
\newcommand{\llvert}{\vert}
\newcommand{\eqref}[1]{(\ref{#1})}
\newproclaim{example}{Example}%[section]
\newproclaim{definition}[example]{Definition}
\newproclaim{convention}[example]{Convention}

\newproclaim{assumption}[example]{Assumption}
\newtheorem{lemma}[example]{Lemma}
\newtheorem{proposition}[example]{Proposition}
\newtheorem{theorem}[example]{Theorem}

\newtheorem{conjecture}{Conjecture}
\newtheorem{question}[conjecture]{Question}
\newtheorem{corollary}[example]{Corollary}

\newproclaim{remark}[example]{Remark}

\newcommand{\bU}{\mathbf{U}}
\newcommand{\bD}{\mathbf{D}}
\newcommand{\bR}{\mathbf{R}}
\newcommand{\bQ}{\mathbf{Q}}
\newcommand{\bT}{\mathbf{T}}
\newcommand{\bX}{\mathbf{X}}
\newcommand{\bY}{\mathbf{Y}}
\newcommand{\bZ}{\mathbf{Z}}
\newcommand{\bW}{\mathbf{W}}
\newcommand{\bV}{\mathbf{V}}
\newcommand{\bI}{\mathbf{I}}
\newcommand{\bzero}{\mathbf{0}}
\newcommand{\bC}{\mathbf{C}}
\newcommand{\bP}{\mathbf{P}}
\newcommand{\bB}{\mathbf{B}}
\newcommand{\bA}{\mathbf{A}}
\newcommand{\1}{\mathbf{1}}
\newcommand{\Z}{\mathbb{Z}}
\newcommand{\spec}{\mathrm{sp}}
\newcommand{\lcal}{\mathcal L}
\newcommand{\ecal}{\mathcal{E}}
\newcommand{\err}{\operatorname{Error}}
\newcommand{\ups}{\upsilon}
\newcommand{\Sin}{\texttt{Sin}}

\newcommand{\eps}{\varepsilon}
\newcommand{\prob}{\mathbb P}
\newcommand{\E}{\mathbb E}
\newcommand{\ci}{i}
\newcommand{\T}{\mathbb{T}}
\makeatother

\begin{document}
\begin{frontmatter}

\title{The top eigenvalue of the random Toeplitz matrix and the sine kernel}
\runtitle{Top eigenvalue of the random Toeplitz matrix}

\begin{aug}
\author[A]{\fnms{Arnab} \snm{Sen}\corref{}\thanksref{t1}\ead[label=e1]{arnab@math.umn.edu}}
\and
\author[B]{\fnms{B\'{a}lint} \snm{Vir\'{a}g}\thanksref{t2}\ead[label=e2]{balint@math.toronto.edu}}
\thankstext{t1}{Supported by the EPSRC Grant EP/G055068/1.}
\thankstext{t2}{Supported by the Canada Research Chair program and the NSERC Discovery Accelerator Grant.}
\runauthor{Sen and Vir\'{a}g}
\address[A]{Departments of Mathematics\\
University of Minnesota\\
127 Vincent Hall, 206 Church St. SE\\
Minneapolis, Minnesota 55455\\
USA\\
\printead{e1}}
\address[B]{Departments of Mathematics and Statistics\\
University of Toronto\\
40 St George St.\\
Toronto, Ontario\\
M5S 2E4, Canada\\
\printead{e2}}

\affiliation{University of Minnesota and University of Toronto}
\runauthor{A. Sen and B. Vir\'{a}g}
\end{aug}

% HISTORY:
\received{\smonth{9} \syear{2011}}
\revised{\smonth{3} \syear{2013}}

% ABSTRACT
%
\begin{abstract}
We show that the top eigenvalue of an $n \times n$ random
symmetric Toeplitz matrix, scaled by $\sqrt{2n\log n}$,
converges to the square of the $2 \to4$ operator norm of the sine
kernel.
\end{abstract}

% KEYWORDS


\end{frontmatter}

%s1 #&#
\section{Introduction}

An $n \times n$ symmetric random Toeplitz matrix is given by
\[
\bT_{n} = %
\lleft[\matrix{ a_0 &
a_1 & \cdots& a_{n-2} & a_{n-1}\vspace*{2pt}
\cr
a_1 & a_0 & a_1& \cdots& a_{n-2}
\vspace*{2pt}
\cr
\vdots& a_1 & a_0 & \ddots& \vdots
\vspace*{2pt}
\cr
a_{n-2} & \vdots& \ddots& \ddots& a_1
\vspace*{2pt}
\cr
a_{n-1} & a_{n-2} & \cdots& a_1 &
a_0 } \rright] %
= \bigl((a_{|i - j|})
\bigr)_{ 0 \le i, j \le n},
\]
where $(a_i)_{ 0 \le i \le n-1}$ is a sequence of
independent random variables. This article establishes
the law of large numbers for the maximum
eigenvalue of this matrix as $n\to\infty$. The study
of deterministic Toeplitz operators has a rich theory. See
the classical book by
\citet{Grenander84} or more recent works by \citet{Bottcher00} and B{\"o}ttcher and
Silbermann (\citeyear{Bottcher99,Bottcher06}). In
contrast, the study of random Toeplitz matrices is a
relatively new field of research. The question of
establishing the limiting spectral distribution of random
Toeplitz matrices with independent entries was first posed
in the review paper by \citet{Bai99}. The answer was
given by \citet{Bryc06} using method
of moments. Since then the study of asymptotic distribution
of eigenvalues of Toeplitz matrices has attracted
considerable attention; for example, see \citet{hammond05,Bose08,Kargin09,Chatterjee09} and the references therein.

The problem of studying the maximum eigenvalue of random Toeplitz
matrices is raised in \citet{Bryc06}, Remark~1.3. \citet{Bose07}
established the law of the large numbers for the spectral norm of
Toeplitz matrix when the entries are i.i.d. with some
\textit{positive} mean and finite variance. But pinpointing the exact
limit for the spectral norm of the Toeplitz matrix when there is
no perturbation, that is, when the entries are mean zero, turned
out to be much more challenging. This is partly due to the fact that,
unlike the Wigner case where the limiting spectral distribution,
the semicircular law has a compact support, and the top eigenvalue
converges to the right endpoint of the support [this was proved
by \citet{Bai88}], the limiting spectral distribution of
Toeplitz matrices has infinite support. As a result, there is no
natural guess to begin with. Another difficulty is that currently
there are no useful estimates available for the trace of high
powers of the Toeplitz matrix $\operatorname{tr}(\bT_n^k)$ when $k =
k(n)$ goes to infinity.

\citet{Meckes07} showed that if the entries have zero mean and
uniformly subgaussian tail, then the expected spectral norm of an
$n \times n$ random Toeplitz matrix is of the order of $\sqrt{n
\log n}$, a significant departure from the standard $\sqrt n$
scaling of the Wigner
case.
\citet{Adamczak10} showed concentration, and more precisely, he proved
that the spectral norm of random Toeplitz matrix normalized by the
expected spectral
norm converges almost surely to $1$ if the entries are
i.i.d. with zero mean and finite variance. \citet{Bose10} gave an
upper bound and a lower bound for the right tail
probability of the spectral norm of Toeplitz matrix scaled
by $n^{1/\alpha}$ when the entries are i.i.d. heavy-tailed random
variables satisfying the following condition. There exist $p, q \ge0$
with $p+q=1$ and a slowly varying function $L(x)$ such that
\[
\lim_{x \to\infty} \frac{\prob\{ X > x \}}{\prob\{|X| > x\}} = p,\qquad
\lim_{x \to\infty}
\frac{\prob\{ X \le-x \}}{\prob\{|X| > x\}} = q
\]
and
\[
\prob\bigl\{|X| > x\bigr\} \sim x^{-\alpha} L(x) \qquad\mbox{as } x \to\infty.
\]

%belonging to the domain of attraction of an $\alpha$-stable random
%variable with index $0 < \alpha< 1$.

Throughout the paper, we will have the following standing
assumption on the entries of our random Toeplitz matrix.

\begin{assumption*}
For each $n$, $(a_i)_{ 0 \le i \le
n-1}$ is an array of independent real random
variables (we suppress the dependence on $n$). There exists
constants $\gamma>2$ and $C$ finite so that for each
variable
\[
\E a_i = 0, \qquad\E a_i^2 = 1 \quad\mbox{and}\quad \E
|a_i|^{\gamma}<C.
\]
\end{assumption*}

Define the integral operator corresponding to the sine
kernel by
\[
\Sin(f) (x):= \int_{\mathbb R} \frac{\sin(\pi(x-y))}{ \pi(x-y)}
f(y) \,dy\qquad \mbox{for } f \in L^2(\mathbb R),
\]
and its $2 \to4$ operator norm as
\[
\|\Sin\|_{2 \to4}:= \sup_{\| f \|_2 \le1} \bigl\| \Sin(f)
\bigr\|_4,
\]
where $\|f\|_p: = (\int_{\mathbb R}
|f(x)|^p)^{1/p} $ denotes the standard $L^p$-norm.
$\|\Sin\|^2_{2 \to4} $ is the solution of the convolution optimization
problem (see the \hyperref[app]{Appendix}),
%
%
%e1 #&#
\begin{equation}
\label{eqopti} \sup\bigl\{\|f\star f\|_2 \dvtx f \mbox{ even of
$L^2$-norm $1$ supported on $[-1/2,1/2]$}\bigr\}.
\end{equation}

For a Hermitian matrix $\bA$, we denote by $\lambda_1(\bA)$ and
$\lambda
_n(\bA)$ the maximum and minimum eigenvalue of $\bA$, respectively. The
following theorem is the main result of our paper.

%th1 #&#
\begin{theorem}\label{thmmain}
Let $\bT_n$ be a sequence of $n \times n$ symmetric random Toeplitz
matrix as defined above with its entries satisfying the above
assumption. Then
\[
\frac{\lambda_1( \bT_n) }{ \sqrt{2 n\log n}} \stackrel{L^\gamma}
{\to} \| \mbox{\em\Sin}
\|^2_{2 \to4} = 0.8288\ldots\qquad\mbox{as } n \to\infty.
\]
\end{theorem}
Recall that a sequence of random variables converges in
$L^p$ to a constant $c$, denoted by $X_n
\stackrel{L^p}{\to} c$, if $\E| X_n - c|^p \to0$.
%
%re2 #&#
\begin{remark}
Note that $L^\gamma$ convergence is as best as we can hope for
in Theorem~\ref{thmmain}. This is because of the fact that
maximum eigenvalue of a symmetric matrix dominates the diagonal
entries. So $\lambda_1(\bT_n) \ge a_0$ and $\E
|\lambda_1(\bT_n)|^p$ can be infinite for any $p> \gamma$. Thus we
cannot expect $L^p$ convergence for $p> \gamma$.
\end{remark}

By symmetry, the same theorem holds for $-\lambda_n$ and so
for the spectral norm $\| \bT_n\|_\spec= \max( \lambda_1,
-\lambda_n )$ as well.

%s1.1 #&#
\subsection{Connection between Toeplitz and circulant matrices}\label{sstoepcirc}
The starting point our analysis of the maximum eigenvalue is a
connection between a Toeplitz matrix and a circulant matrix twice
its size.

Observe that $\bT_n$ is the $n \times n$ principal submatrix of a
$2n \times2n$ circulant matrix $\bC_{2n}= (b_{j-i \mathrm{ mod
} 2n})_{0 \le i, j \le2n-1}$, where $b_j= a_j$ for $ 0 \le j <
n$ and $b_j = a_{2n -j}$ for $n< j < 2n$ (choice of $b_n$ is irrelevant
at this point and it will be set later). We hope to relate the
spectrum of Toeplitz matrix to that of the present circulant
matrix twice its size, which can be easily diagonalized as follows:
\[
(2n)^{-1/2}\bC_{2n} = \bU_{2n} \bD^{\dagger}_{2n}
\bU_{2n}^*,
\]
where $\bU_{2n}$ is the discrete Fourier transform, that is,
a unitary matrix given by
\[
\bU_{2n}(j, k) = \frac{1}{\sqrt{2n}}\exp\biggl( \frac{2 \pi
ijk}{2n}
\biggr),\qquad  0 \le j, k \le2n-1,
\]
and $\bD^{\dagger}_{2n}$ is a diagonal matrix with
%
%
%e2 #&#
\begin{eqnarray}
 {\bigl(\bD^{\dagger}_{2n}
\bigr)}_{j,j}&=& \frac{1}{\sqrt{2n}} \sum_{k=0}^{2n-1}
b_k \exp\biggl( \frac{2 \pi\ci jk}{2n} \biggr)
\nonumber
\\[-8pt]
\\[-8pt]
\nonumber
& =& \frac{1}{\sqrt{2n}}
\Biggl[ a_0 + (-1)^jb_n + 2\sum
_{k=1}^{n-1} a_k \cos\biggl(
\frac{2
\pi jk}{2n} \biggr) \Biggr].
\end{eqnarray}
Clearly the $j$ and $2n-j$ entries of $\bD^{\dagger}_{2n}$ agree for
all $n < j < 2n$.
If we write
\[
\bQ_{2n} = %
\pmatrix{ \bI_{n} &
\bzero_n\vspace*{2pt}
\cr
\bzero_n & \bzero_n }
,
\]
then $\bT_n$ and $\bQ_{2n} \bC_{2n} \bQ_{2n} $ have the same nonzero
eigenvalues by our observation. Moreover, the matrix $(2n)^{-1/2}
\bQ_{2n} \bC_{2n} \bQ_{2n} $ has the same eigenvalues as its
conjugate
\[
(2n)^{-1/2} \bU^*_{2n}\bQ_{2n} \bC_{2n}
\bQ_{2n}\bU_{2n}= \bP_{2n} \bD^{\dagger}_{2n}
\bP_{2n},
\]
where
%
%
%e3 #&#
\begin{equation}
\label{pdef} \bP_{2n}:= \bU^*_{2n} \bQ_{2n}
\bU_{2n}.
\end{equation}
Consequently, we have a useful
representation of the maximum eigenvalue of the Toeplitz
matrix
%
%
%e4 #&#
\begin{equation}
\lambda_1\bigl(n^{-1/2}\bT_n\bigr) = \sqrt{2}
\lambda_1\bigl(\bP_{2n} \bD^{\dagger
}_{2n}
\bP_{2n}\bigr) \label{eqtoepcirc}
\end{equation}
as long as the right-hand side is not zero. We point out here
that the matrix $\bQ_{2n} \bC_{2n} \bQ_{2n} $ (and so $\bP_{2n}
\bD^{\dagger}_{2n} \bP_{2n}$) does not depend on the value of
$b_n$, so we may replace it with an independent copy $a_n$ of $a_0$. In
addition, as we will show in Lemma~\ref{lemdiagchange} we can replace
$a_0$ by $\sqrt{2} a_0$ in $D_{2n}$ without changing the asymptotics.
Dropping the subscript $2n$, we will study the matrix
$\bP\bD\bP$ as before and the entries of the diagonal matrix $ \bD=
\operatorname{diag}(d_0, d_1, \ldots, d_{2n-1} )$ given by
\[
d_j = \frac{1}{\sqrt{2n}} \Biggl[ \sqrt{2} a_0 +
(-1)^j\sqrt{2} a_n + 2\sum_{k=1}^{n-1}
a_k \cos\biggl( \frac{2 \pi jk}{2n} \biggr) \Biggr],\qquad 0 \le j < 2n.
\]
The reason for choosing the ``right'' variance for diagonal and the
``right'' auxiliary variable $b_n$ is that now the variables $d_j, 0
\le j \le n$ become uncorrelated; see Lemma~\ref{lemdjmeancov}. Thus
in the special case when $\{a_j \dvtx0 \le j \le n\} $ are i.i.d.
Gaussian random variables with mean $0$ and variance $1$, it follows
that $\{ d_j\dvtx0 \le j \le n\}$ are again independent Gaussian with mean
$0$ and have variance $1$ except for $d_0$ and~$d_n$, which have
variance $2$.

We have thus reduced our problem to studying the maximum
eigenvalue of the matrix $\bP\bD\bP$ where
$\bP$ is a deterministic Hermitian projection operator, and
$\bD$ is a random multiplication operator with defined
on $\mathbb C^{2n}$ with uncorrelated entries.
One can view this representation as a discrete
analogue for Toeplitz operators defined on the Hardy space
$\mathcal H^2$.

%But it is not clear how one can
%take limit so that the finite dimensional operators
%$\bP$ and $\bD^{\dagger}$ would converge to an appropriate
%infinite dimensional projection operator and a random
%multiplication operator respectively and moreover, the
%spectral norm would be continuous with respect to such limit.

%The spectral properties of both $\bP$ and $\bD^{\dagger}$
%are easy to analyze in isolation but the analysis of
%operator $\bP\bD^{\dagger} \bP$ is more challenging
%because the operators $\bP$ and $\bD^{\dagger}$ do not
%commute.
%
In the general case, to analyze the lower bound, we need to use an
invariance principle.
The difficulty is that the top eigenvalue does not come from the usual
central limit theorem regime of $\bD$, but from moderate deviations. We
overcome this by extending the invariance principle of \citet
{chatterjee05} (based on Lindeberg's approach to the CLT) to the realm
of moderate deviations.

%However, this is not true for general Toeplitz
%matrix, a substantial technical difficulty in proving the
%lower bound in Theorem~\ref{thmmain}.

%s1.2 #&#
\subsection{Heuristics and conjectures}\label{sshc}

We first give a heuristic description of the origin of the limiting
constant. Since $\bP$ is a convolution with a decaying function,
the matrix $\bP\bD\bP$ in many ways behaves
a like the diagonal matrix $\bD$ itself. In particular, its top
eigenvalue comes
from a few nearby extreme values of $\bD$. We can partition of the interval
into
short enough sections $J$, and
show that the top eigenvalue is close to the maximum top eigenvalue of
blocks $\bP[J]\bD[J]\bP[J]$. This likes to be
large when the entries of $\bD[J]$ are large. By large deviations
theory, the real constraint on the best $\bD[J]$ is that the $\ell
^2$-norm the vector of entries
is bounded by $\sqrt{2 \log n}$.

Now the top eigenvalue
is an $\ell^2$ optimization problem, but now we have another
$\ell^2$ optimization over the entries of $\bD[J]$. Thus an $\ell
^4$-norm appears, and that the solution
of these two problems together are asymptotically given by
the $2\to4$ norm of the limiting operator of $\bP_{2n}$, which we
then relate to the well-known sine kernel. The relation is
natural, since $\bP_{2n}$ and the sine kernel are both projections to
an interval conjugated by a Fourier transform.

Before we prove Theorem~\ref{thmmain}, let us state some conjectures
and open
questions. Each of these conjectures
can be split into parts (a) the Gaussian case, and (b) the general case
where suitable moment conditions have to be imposed.

%co1 #&#
\begin{conjecture}
Let $v_n$ be the top eigenvector of $\bP\bD^{\dagger} \bP$. Then there
exist random integers $K_n$ so that for each $i\in\mathbb Z$,
we have $v_n(K_n+i) \to\hat g(i)$, where $\hat g$ is the Fourier
transform of the function $g(x) = \sqrt2 f( 2x -1/2)$, and f is the
(unique) optimizer in \eqref{eqopti}.
\end{conjecture}

%co2 #&#
\begin{conjecture}
With high probability, all eigenvectors of $\bP\bD^{\dagger} \bP$ are
localized: for each eigenvector, there exists a set of size $n^{o(1)}$
that supports $1-o(1)$ proportion
of the $\ell^2$-norm.
\end{conjecture}

%qu3 #&#
\begin{question}
What is the behavior of $n^{-1/2} \lambda_1(\bT_n)-\sqrt{2\log n}\|
\mbox
{\em\Sin}\|^2_{2 \to4}$?
\end{question}

%co4 #&#
\begin{conjecture}
The top of the spectrum of $\bT_n$, suitably shifted, and normalized,
converges to a Poisson process with intensity $c e^{-\eta x}$ for some
$c,\eta>0$.
\end{conjecture}

Related to this, we have the following.

%co5 #&#
\begin{conjecture}
The top eigenvalue of $\bT_n$, suitably shifted and normalized,
has a limiting Gumbel distribution.
\end{conjecture}

%co6 #&#
\begin{conjecture}
The eigenvalue process of $\bT_n$, away from the
edge, after suitable normalization, converges to a standard Poisson
point process on $\mathbb R$.
\end{conjecture}

%s1.3 #&#
\subsection{\texorpdfstring{The proof of Theorem \protect\ref{thmmain}}
{The proof of Theorem 1}}

In this section, we break the proof of Theorem~\ref{thmmain} into its
components. This also serves as
a guide to the rest of the paper.

\begin{pf*}{Proof of Theorem~\ref{thmmain}}
Consider the discrete Fourier transform matrix,
\[
\bU(j, k) = \frac{1}{\sqrt{2n}}\exp\biggl( \frac{2 \pi ijk}{2n}
\biggr),\qquad 0 \le j, k
\le2n-1
\]
and $\bD^{\dagger}$, a $2n\times2n$ diagonal matrix so that for
$j=0,\ldots, 2n-1$
%
%
%e5 #&#
\begin{eqnarray}
\label{djformula} \bD^{\dagger}(j,j)&=& \frac{1}{\sqrt{2n}} \sum
_{k=0}^{2n-1} b_k \exp\biggl(
\frac{2 \pi\ci jk}{2n} \biggr)
\nonumber
\\[-8pt]
\\[-8pt]
\nonumber
&=& \frac{1}{\sqrt{2n}} \Biggl[ a_0 +
(-1)^jb_n + 2\sum_{k=1}^{n-1}
a_k \cos\biggl( \frac{2 \pi
jk}{2n} \biggr) \Biggr].
\end{eqnarray}
Let
\[
\bQ= %
\pmatrix{ \bI_{n} & \bzero_n
\vspace*{2pt}
\cr
\bzero_n & \bzero_n} %
, \qquad \bP:=
\bU^*\bQ\bU.
\]
Then, as argued in Section~\ref{sstoepcirc} we have
the representation \eqref{eqtoepcirc}
\[
\lambda_1\bigl(n^{-1/2}\bT_n\bigr) = \sqrt{2}
\lambda_1\bigl(\bP\bD^{\dagger} \bP\bigr)
\]
as long as the right-hand side is positive.
In formula \eqref{djformula} for diagonal entries of $\bD^{\dagger}$,
we replace $a_0$ by $\sqrt2 a_0$ (this is legal via Lemma~\ref
{lemdiagchange}) and choose $b_n = \sqrt2 a_n$ where $a_n$ is an
identical copy of $a_0$ independent of $(a_i)_{0 \le i < n}$ (recall
$\bP\bD^{\dagger}\bP$ does not depend on $b_n$) to obtain a new
diagonal matrix $ \bD= \operatorname{diag}(d_0, d_1, \ldots,
d_{2n-1} )$
given by
\[
d_j = \frac{1}{\sqrt{2n}} \Biggl[ \sqrt{2} a_0 +
(-1)^j\sqrt{2} a_n + 2\sum_{k=1}^{n-1}
a_k \cos\biggl( \frac{2 \pi jk}{2n} \biggr) \Biggr], \qquad 0 \le j < 2n.
\]
The reason for choosing the ``right'' variance for diagonal and the
``right'' auxiliary variable $b_n$ is that now the variables $d_j, 0
\le j \le n$ become uncorrelated; see Lemma~\ref{lemdjmeancov}. Thus
in the special case when $\{a_j \dvtx0 \le j \le n\} $ are i.i.d.
Gaussian random variables with mean $0$ and variance $1$, it follows
that $\{ d_j\dvtx0 \le j \le n\}$ are again independent Gaussian with mean
$0$ and have variance $1$ except for $d_0$ and $d_n$ who have variance $2$.

In Lemma~\ref{lemdiagchange} we justify working with $\bP\bD^\dagger
\bP$ by showing that
\[
\mathbb E \bigl|\lambda_1(\bP\bD\bP)-\lambda_1\bigl(\bP
\bD^\dagger\bP\bigr)\bigr|^\gamma= o(\sqrt{\log n}).
\]
In Section~\ref{sectrunc}, Corollary~\ref{cortrunc} we show that we
can assume that the $a_n$ are bounded by $n^{1/\gamma}$. In Lemma~\ref
{lemtightness}, we establish tightness, so it suffices to show
convergence in probability. In Section~\ref{secblock}, equation \eqref
{deps} we introduce a sparse version $\bD^\eps$ of the diagonal matrix
$\bD$, by considering the set
\[
S=\bigl\{0\le j\le2n-1 \dvtx|d_j|\ge\eps\sqrt{2\log n}\bigr\},
\]
and setting $(\bD^\eps)_{jj}=d_j{\mathbf1}_{j\in S}$. Then we show
that the eigenvalues are close,
\[
\bigl\llvert\lambda_1( \bP\bD\bP) - \lambda_1\bigl( \bP
\bD^{\eps} \bP\bigr)\bigr\rrvert\le\eps\sqrt{2\log n}.
\]

The matrix $\bD^\eps$ is sparse enough that the whole question can be
reduced to a block-diagonal version, where the blocks are determined by
a random partition $\Lambda$ of $\{1,\ldots,2n\}$.
This is done in Lemma~\ref{lemblockform}: with high probability,
\[
\Bigl\llvert\lambda_1 \bigl( \bP\bD^{\eps} \bP\bigr) -
\max_{ J\in\Lambda: J
\cap S \ne\varnothing} \lambda_1 \bigl( \bP[J ]
\bD^{\eps} [J] \bP[J ] \bigr)\Bigr\rrvert= O(1),
\]
where $\bP[J]$ refers to the minor of $\bP$ corresponding to the subset
of indices $J$. This is guaranteed by a careful choice of partition
$\Lambda$, which ensures that the blocks $J$ with $J\cap S\neq
\varnothing$ are sufficiently far apart, so that the interaction between
them is negligible. This interaction comes from off-diagonal elements
of the matrix $\bP$, whose entries decay with the distance from the diagonal.

The main idea of the last part of the proof is explained in Section~\ref{sshc}.
We proceed along those lines. In Propositions \ref{proupperbound} and
\ref{prolowerbound} (Sections~\ref{supper}, \ref{slower}) we then
establish the asymptotic upper and lower bounds for the for the block
diagonal version.
Together, they give that
with high probability, there exists $\eps_n\to0$ so that %(xxx I had
%to change this
%and Prop 15)
%
\[
\max_{ J\in\Lambda: J \cap S \ne\varnothing} \lambda_1 \bigl(
\bP[J ] \bD
^{\eps_n} [J] \bP[J ] \bigr) = \sqrt{2\log n}\bigl(\| \Pi
\|_{2 \to4}^2+o(1)\bigr),
\]
where $\Pi$ is the $n\to\infty$ limit of $\bP$ introduced in \eqref
{eqPidef}. Finally, in the \hyperref[app]{Appendix}, Lemma~\ref{lemsinepiconn},
we identify $\| \Pi\|_{2 \to4}^2 = \frac{1}{\sqrt2} \| \mbox{$\Sin
$}\|
_{2 \to4}^2$. This completes the proof of Theorem~\ref{thmmain}.
\end{pf*}

%s1.4 #&#
\subsection{Notation}
We write that a sequence of events $(\ecal_n)_{n \ge1}$
occurs \textit{with high probability} when $\prob\{
\ecal_n\} \to1$. Let $\ell^2(\mathbb C)$ [resp., $\ell^2(\mathbb
R)$] be the space of square summable
sequences of complex (resp., real) numbers indexed by
$\Z$. For square matrix $\mathbf{A}$ and a subset $T$ of
the index set, we denote by $\mathbf{A}[T]$, the principal
submatrix of $\mathbf{A}$ which is obtained by keeping
those rows and columns of $\mathbf{A}$ whose indices
belong to $T$. We consider $n$ as an asymptotic parameter
tending to infinity. We will use the notation $f(n)
=\Omega(g(n))$ or $g(n) =O(f(n))$ to denote the bound $
g(n) \le C f(n)$ for all sufficiently large $n$ and for
some constant $C$. Notation such as $f(n) =\omega(g(n))$
or $g(n) =o(f(n))$ means that $g(n)/f(n) \to0$ as $n \to
\infty$. We write $f(n) = \Theta(g(n))$ if both $ f(n) =
O(g(n))$ and $g(n) = O(f(n))$ hold.

%s2 #&#
\section{Truncation and tightness}\label{sectrunc}

%s2.1 #&#
\subsection{Truncation and changing the diagonal term} Let $n_0$ be
sufficiently large number. For $n \ge n_0$,
define two arrays of truncated random variables~by
\[
\tilde a_i =\tilde a_i^{(n)}=:
a_i \1_{ \{|a_i| \le
n^{1/\gamma} \}} - \E[ a_i \1_{ \{|a_i| \le n^{1/\gamma} \}}
],\qquad 0 \le i \le n-1
\]
and
\[
\bar a_i =\bar a_i^{(n)}:= \operatorname{Var}(
\tilde a_i)^{-1/2} \tilde a_i, \qquad 0 \le i \le
n-1.
\]
Note that to define $\bar a_i^{(n)}$, we need that $\operatorname{Var}(
\tilde a_i^{(n)}) >0$, which holds for sufficiently large $n$.

We sometimes write $\bT_n(a)$ for
$\bT_n$ to emphasize its dependence on the underlying sequence of
random variables $(a_i)_{0 \le i \le n-1}$. Thus $\bT_n(\tilde a)$
and $\bT_n(\bar a)$ denote the Toeplitz matrices built with random
variables $(\tilde a_i)_{0 \le i \le n-1}$ and $(\bar a_i)_{0 \le
i \le n-1}$, respectively.

The next lemma says that the above truncation and the rescaling of the
underlying random variables has a negligible effect in the study of the
maximum eigenvalue of Toeplitz matrix.
%
%le3 #&#
\begin{lemma} \label{lemtrunc}
We have, as $ n \to\infty$,
\begin{longlist}[(a)]
\item[(a)]
\[\vspace*{-6pt}
\frac{\lambda_1(\bT_n (a) ) - \lambda_1(\bT_n(\tilde a) ) }{ \sqrt
{ n
\log n} } \stackrel{L^\gamma} {\to} 0,
\]
\item[(b)]
\[\vspace*{-6pt}
\frac{\lambda_1(\bT_n (\tilde a) ) - \lambda_1(\bT_n(\bar a) ) }{
\sqrt{ n \log n} } \stackrel{L^\gamma} {\to} 0.
\]
\end{longlist}
\end{lemma}
\begin{pf}
Define, for $n
\ge n_0$,
\[
\hat a_i:= a_i \1_{ \{|a_i| \le n^{1/\gamma} \}},\qquad 0 \le i \le n-1.
\]
Recall that for a matrix $\bA$, its spectral norm satisfies
%
%
%e6 #&#
\begin{equation}
\label{eqspecnormbound} \| \bA\|_\spec^2 \le\max
_k \sum_{ l}\bigl|\bA(k, l)\bigr| \times
\max_l \sum_{ k}\bigl|\bA(k, l)\bigr|.
\end{equation}
In the special case when $\bA$ is Hermitian, the above bound reduces to
%
%
%e7 #&#
\begin{equation}
\label{eqhermitianspecnormbound} \| \bA\|_\spec\le\max
_k \sum_{ l}\bigl|\bA(k, l)\bigr|.
\end{equation}
Then we have
%
%
%e8 #&#
\begin{eqnarray}
\label{eqtrunc1} \biggl\llvert\frac{\lambda_1(\bT_n (a) ) -
\lambda_1(\bT_n(\hat a) ) }{ \sqrt{ n \log n} } \biggr\rrvert&\le&
\frac{\| \bT_n (a) - \bT_n(\hat a) \|_\spec
}{ \sqrt{ n \log n} }
\nonumber
\\[-8pt]
\\[-8pt]
\nonumber
& \le&\frac{2 \sum_{i=0}^{n-1} |a_i| \1_{ \{ |a_i|
> n^{1/\gamma} \}}}{ \sqrt{ n \log n} },
\end{eqnarray}
which follows from bound \eqref{eqhermitianspecnormbound}
and from the fact that the $\ell^1$ norm of the each row of the Toeplitz
matrix $\bT_n (a) - \bT_n(\hat a) $ is bounded by $2
\sum_{i=1}^n |a_i| \1_{ \{ |a_i| > n^{1/\gamma} \}}$. We quote a
standard moment bound for sum of independent nonnegative random
variables, commonly known as Rosenthal's inequality in the
literature [see, e.g., \citet{Latala97}, Corollary~3] which says
that if $\xi_1, \ldots, \xi_n$ are independent nonnegative random
variables and $p \ge1$, then there exists a universal constant $C_p$ such
that
\[
\E\Biggl[ \Biggl(\sum_{i=1}^n
\xi_i \Biggr)^p \Biggr] \le C_p \max\Biggl(
\Biggl(\sum_{i=1}^n \E[
\xi_i] \Biggr)^p, \sum_{i=1}^n
\E\bigl[\xi_i^p \bigr] \Biggr).
\]
Clearly, $\E[|a_i|^\gamma\1_{ \{ |a_i| > n^{1/\gamma} \}} ] \le
\E[|a_i|^\gamma] \le C$. On the other hand, by H\"{o}lder's
inequality,
\[
\E\bigl[ |a_i| \1_{ \{ |a_i| > n^{1/\gamma} \}} \bigr] \le\bigl(\E
|a_i|^\gamma\bigr)^{1/\gamma} \cdot\bigl(\prob\bigl\{
|a_i|^\gamma> n \bigr\}\bigr)^{1 -
1/ \gamma},
\]
and by Markov's inequality, this is bounded above by
%
%
%e9 #&#
\begin{equation}
\label{tailbound} \bigl(\E|a_i|^\gamma
\bigr)^{1/\gamma} \bigl(\E|a_i|^\gamma/n
\bigr)^{1 - 1/\gamma} = \E\bigl[|a_i|^\gamma\bigr]
n^{-1 + 1/\gamma} \le Cn^{-1 +
1/\gamma}.
\end{equation}
Therefore, by Rosenthal's inequality,
%
%
%e10 #&#
\begin{equation}
\label{eqrosenthalbd} \E\Biggl[ \sum
_{i=0}^{n-1} |a_i| \1_{ \{ |a_i| > n^{1/\gamma} \}}
\Biggr]^\gamma= O(n).
\end{equation}
Combining \eqref{eqtrunc1} and \eqref{eqrosenthalbd},
we obtain
\[
\E\biggl\llvert\frac{\lambda_1(\bT_n (a) ) - \lambda_1(\bT
_n(\hat a) ) }{
\sqrt{ n \log n} } \biggr\rrvert^\gamma=
\frac{O(n)} { (n \log n)^{\gamma/2}
} \to0.
\]
Next we see that
\begin{eqnarray*}
\E\biggl\llvert\frac{\lambda_1(\bT_n (\hat a) ) - \lambda
_1(\bT_n(\tilde a) ) }{ \sqrt{ n \log n} } \biggr\rrvert^\gamma&\le&
\frac
{\E[ \| \bT_n (\hat a) - \bT_n(\tilde a) \|_\spec^\gamma]}{ (n
\log
n)^{\gamma/2} } \\
& \le&\frac{ (2 \sum_{i=0}^{n-1} \E|a_i| \1_{ \{
|a_i| > n^{1/\gamma} \}
})^\gamma}{ (n \log n)^{\gamma/2} } = \frac{O(n)}{ (n \log
n)^{\gamma
/2}}\to0,
\end{eqnarray*}
which completes the proof of part (a).

For part (b) we want to bound
\[
\biggl\llvert\frac{\lambda_1(\bT_n (\tilde a) ) - \lambda_1(\bT
_n(\bar a) )
}{ \sqrt{ n \log n} } \biggr\rrvert\le\frac{ \|\bT_n (\tilde a) -
\bT
_n(\bar a) \|_\spec}{ \sqrt{ n \log n} } \le
\sqrt{2} \cdot\frac{ \| \bP(\bD^{\dagger} (\tilde a) - \bD(\bar
a))\bP\|_\spec}{ \sqrt{ \log n} }.
\]
Here we use representation
\eqref{eqtoepcirc}. Since
$\bP$, being Hermitian projection matrix, has spectral norm
equal to one, with $a^*=\tilde a -\bar a$ we have
\[
\bigl\| \bP\bigl(\bD^{\dagger} (\tilde a) - \bD(\bar a)\bigr)\bP
\bigr\|_\spec=\bigl\| \bP\bD^{\dagger}\bigl(a^*\bigr) \bP\bigr\|_\spec
\le\bigl\| \bD^{\dagger}\bigl(a^*\bigr)\bigr\|_\spec= \max
_{0
\le j \le n} \bigl|d_j\bigl( a^*\bigr) \bigr|.
\]

We have
\[
1 - \operatorname{Var}(\tilde a_i) = \E\bigl[ a_i^2
\1_{ \{|a_i| > n^{1/\gamma} \}
} \bigr] + \bigl(\E[ a_i \1_{ \{|a_i| > n^{1/\gamma} \} } ]
\bigr)^2.
\]
We apply \eqref{tailbound}, and the similarly derived
\[
\E\bigl[ a_i^2 \1_{ \{|a_i| > n^{1/\gamma} \} } \bigr] \le
Cn^{-1 + 2/\gamma}
\]
to get $ 1 - \operatorname{Var}(\tilde a_i) = O(n^{-1 + 2/\gamma})$ uniformly
in $i$.
Note that $a_i^* = (1-\break  \operatorname{Var}(\tilde a_i)^{-1/2} )\tilde a_i$
which implies $\E a_i^* = 0$ and $\operatorname{Var}(a^*_i) = (1 -
\operatorname
{Var}(\tilde a_i)^{1/2} )^2 = O(n^{-2 + 4/\gamma})$ and $|a_i^*| \le
4n^{1/\gamma}$, for $n \ge n_0$.

Using the union bound and the identity $\E[ X] = \int_0^\infty\prob
\{
X> t\} \,dt$ which holds for any nonnegative random variable $X$, we can write
%
%
%e11 #&#
\begin{eqnarray}\label{eqgammanormneg}
\E\Bigl[\max_{0 \le j \le n} \bigl|d_j\bigl( a^*\bigr)
\bigr|^\gamma\Bigr] &\le&1+ \int_{1}^\infty\prob
\Bigl\{ \max_{ 0 \le j \le n} \bigl|d_j\bigl( a^*\bigr)
\bigr|^\gamma> t \Bigr\} \,dt
\nonumber
\\[-8pt]
\\[-8pt]
\nonumber
&\le&1+ (n+1) \max_{ 0 \le j \le n} \int_{1}^\infty
\prob\bigl\{ \bigl|d_j\bigl( a^*\bigr) \bigr| > t^{1/\gamma} \bigr\} \,dt.
\end{eqnarray}
If $\xi_1, \xi_2, \ldots, \xi_n$ be independent mean zero random
variables, uniformly\break bounded by $M$, then the classical Bernstein
inequality gives the following tail bound for their sum:
\[
\prob\Biggl\{ \sum_{ k=1}^n
\xi_k > t\Biggr\} \le\exp\biggl( - \frac{t^2/2}{\sum_{k=1}^n
\operatorname{Var}(\xi_k) + M t/3 } \biggr)\qquad \forall
t>0.
\]
Using Bernstein's inequality, we obtain
\begin{eqnarray*}
\eqref{eqgammanormneg} &\le&1+ (n+1) \int
_1^\infty2 \exp\biggl( - \frac{ t^{2/\gamma}(\sqrt n)^2}{ n \cdot
O(n^{-1 + 2/\gamma}) +
t^{1/\gamma}\sqrt{n} O(n^{1/\gamma})} \biggr)
\,dt
\nonumber\\
&\le&1+ 2(n+1) \int_1^\infty\exp\bigl( -
t^{1/\gamma} \cdot\Omega\bigl(n^{1/2 - 1/\gamma}\bigr) \bigr) \,dt
\\
&= &1 + n \cdot
n^{1/\gamma-1/2 }e^{ -
\Omega(n^{1/2 - 1/\gamma})}.\nonumber %&\le1+ 2(n+1) O(n^{ - \gamma/2 +1 })
\end{eqnarray*}
Hence, $ (\log n)^{-1/2} \| \bD(a^*)\| \stackrel{L^\gamma}{\to} 0$ and
so, $ \llvert\frac{\lambda_1(\bT_n (\tilde a) ) - \lambda_1(\bT
_n(\bar
a) ) }{ \sqrt{ n \log n} } \rrvert\stackrel{L^\gamma}{\to} 0$. This
completes the proof of part (b) of the lemma.
\end{pf}
%
%de4 #&#
\begin{definition}\label{defchangevariance}
Let $\bT_n^\circ$ be the symmetric Toeplitz matrix which has $\sqrt{2}
a_0$ on its diagonal instead of $a_0$.
\end{definition}
%
%le5 #&#
\begin{lemma} \label{lemdiagchange}
We have, as $ n \to\infty$,
\[
\frac{\lambda_1(\bT^\circ_n ) - \lambda_1(\bT_n) }{ \sqrt{ n
\log n} } \stackrel{L^\gamma} {\to} 0.
\]
\end{lemma}
\begin{pf}
The proof is immediate from the following fact:
\[
\E\bigl\| \bT^\circ_n - \bT_n\bigr \|^\gamma_\spec
\le(\sqrt2 -1)^\gamma\E\bigl[ |a_0|^\gamma\bigr] =
O(1). %\qedhere
\]
\upqed\end{pf}

%co6 #&#
\begin{corollary}\label{cortrunc}
It suffices to prove Theorem~\ref{thmmain} for the symmetric random
Toeplitz matrix $\bT_n^\circ$ defined in Definition~\ref
{defchangevariance} where the random variables $a_i$ are independent
mean zero, variance one and bounded by $2n^{1/\gamma}$.
\end{corollary}
\begin{pf}
The proof is immediate from Lemmas~\ref{lemtrunc} and~\ref
{lemdiagchange}.
\end{pf}

Following \eqref{eqtoepcirc}, we can write
%
%
%e12 #&#
\begin{equation}
\label{eqToepcircconnectionrightvariance}
\lambda_1\bigl(n^{-1/2}\bT_n^\circ
\bigr) = \sqrt{2} \lambda_1(\bP\bD\bP),
\end{equation}
with $\bP$ as before and the entries of the diagonal matrix $ \bD=
\operatorname{diag}(d_0, d_1, \ldots,\break d_{2n-1} )$ given by
\[
d_j = \frac{1}{\sqrt{2n}} \Biggl[ \sqrt{2} a_0 +
(-1)^j\sqrt{2} a_n + 2\sum_{k=1}^{n-1}
a_k \cos\biggl( \frac{2 \pi jk}{2n} \biggr) \Biggr],\qquad 0 \le j < 2n,
\]
where $b_n:= \sqrt2 a_n, a_n$ being an independent copy of $a_0$.
The reason for choosing the ``right'' variance for diagonal and the
``right'' auxiliary variable $b_n$ is that now the variables $d_j, 0
\le j \le n$ become uncorrelated; see Lemma~\ref{lemdjmeancov}. Thus
in the special case when $\{a_j \dvtx0 \le j \le n\} $ are i.i.d.
Gaussian random variables with mean~$0$ and variance $1$, it follows
that $\{ d_j\dvtx0 \le j \le n\}$ are again independent Gaussian with mean
$0$ and have variance $1$ except for $d_0$ and $d_n$ which have
variance $2$.
%s2.2 #&#
\subsection{Tightness}

%le7 #&#
\begin{lemma} \label{lemtightness}
For each $n \ge1$, let $a_0, a_1, \ldots, a_{n-1}$ be a sequence of
independent random variables that have mean zero, variance one and are
bounded by $n^{1/\gamma}$. For any $p> 0$, we have
\[
\sup_{ n \ge1 }\E\biggl[ \biggl(\frac{ \lambda_1(\bT_n^\circ
)}{\sqrt{2n \log n}}
\biggr)^p \biggr] < \infty.
\]
\end{lemma}
\begin{pf}
The proof is a direct application of Bernstein's inequality similar to
what we did in the proof of part (b) of Lemma~\ref{lemtrunc}.
From representation \eqref{eqToepcircconnectionrightvariance}, we
know that
$n^{-1/2}\lambda_1(\bT_n^\circ)$ is bounded above by $\sqrt2 \cdot
\max_{0 \le j \le n} |d_j| $. Therefore, for any $\alpha>0$,
\begin{eqnarray*}
\E\biggl[ \biggl(\frac{ \lambda_1(\bT_n^\circ)}{\sqrt{2n \log
n}} \biggr)^p \biggr] &\le&\E
\biggl[ \biggl( \frac{\max_{0 \le j \le
n}|d_j| }{ \sqrt{\log n}} \biggr)^p \biggr]
\\
&\le&\alpha+ (n+1) \max_{ 0 \le j \le n-1} \int_\alpha^\infty
\prob\biggl\{ \frac{|d_j|}{\sqrt{\log n}} > t^{1/p} \biggr\} \,dt.
\end{eqnarray*}
By the Bernstein inequality, for $n$ sufficiently large,
\[
\max_{ 0 \le j \le n} \prob\biggl\{ \frac{|d_j|}{\sqrt{\log n}} >
t^{1/p} \biggr\} \le2\exp\biggl(- \frac{(1/2) \log n \cdot
t^{2/p}}{2+ (1/3) \sqrt2 n^{1/\gamma- 1/2} \sqrt{ \log n} \cdot
t^{1/p} } \biggr),
\]
which implies that there exists a constant $c>0$ such that for each $n$
and each $0 \le j \le n$,
\[
\int_\alpha^\infty\prob\biggl\{ \frac{|d_j|}{\sqrt{\log n}} >
t^{1/p} \biggr\} \,dt \le\int_\alpha^{t_n}
\exp\bigl(- c t^{2/p}\cdot\log n \bigr) \,dt + \int_{t_n}^\infty
\exp\bigl(- c t^{1/p} \bigr) \,dt,
\]
where $t_n:= n^{p(1/2 -1/\gamma)}(\log n)^{-p/2}$. This particular
choice of $t_n$ is governed by the fact that $2+ \frac13 \sqrt2
n^{1/\gamma- 1/2} \sqrt{ \log n} \cdot t_n^{1/p} = O(1)$.
The second integral above goes to zero faster than any polynomial power
of $n$, whereas by choosing $\alpha$ sufficiently large we can make the
first integral $O(n^{-1})$. The claim of the lemma follows.
\end{pf}

%s3 #&#
\section{Reduction to block diagonal form}\label{secblock}

%s3.1 #&#
\subsection{Some facts about $\bP$} \label{subsecpropP}
By definition \eqref{pdef} $\bP\dvtx\mathbb C^{2n} \to\mathbb
C^{2n}$ is
a Hermitian projection matrix. The action of operator $\bP$ can be
described by the composition of the following three maps: For $x \in
\mathbb C^{2n}$, we first take discrete Fourier transform of $x$, then
project it to the first $n$ Fourier frequencies and finally do the
inverse discrete Fourier transform.

The entries of $\bP$ are given by
\[
\bP(k, l) = \cases{ %
 \displaystyle\frac12, &\quad $\mbox{for } k=l,$
\vspace*{2pt}\cr
0, & \quad$\mbox{for } k \ne l, |k-l| \mbox{ is even},$
\vspace*{2pt}\cr
\displaystyle\frac{1}{n} \times\frac{1}{1 - \exp( - {2 \pi\ci
(k-l)}/{(2n)} )}, &\quad $\mbox{for } |k- l| \mbox{ is
odd}.$}
\]
Note that $\bP(k, l)$ is a function of $(k-l)$ only and that
\[
\bigl|\bP(k, l)\bigr| \le C_1 / \min\bigl( |k-l|, 2n- |k-l| \bigr),\qquad k \ne l
\]
for some constant $C_1$.
Hence, the maximum of $\ell^1$ norms of the rows or the columns of
$\bP
$ has the following upper bound:
%
%
%e13 #&#
\begin{equation}
\label{eqProwsum} \max_{k} \sum
_{l=0}^{2n-1} \bigl|\bP(k, l)\bigr| \le C_2 \log n,\qquad
\max_{l} \sum_{k=0}^{2n-1}
\bigl|\bP(k, l)\bigr| \le C_2 \log n,
\end{equation}
where $C_2$ is some suitable constant.

\subsubsection*{Limiting operator for $\bP$} Let $\T$ be unit circle
parametrized by $\T= \{ e^{ 2 \pi\ci x}\dvtx x \in(-1/2, 1/2] \}$ and
$L^2(\T):= \{ f\dvtx[-1/2, 1/2] \to\mathbb C \dvtx f(-1/2) =
f(1/2)\break
\mbox{and } \int_{-1/2}^{1/2} |f(x)|^2 \,dx < \infty\}$. We define an
projection operator $\Pi\dvtx\ell^2(\mathbb C) \to\ell^2(\mathbb
C)$ as
the composition of the following operators:
%
%
%e14 #&#
\begin{equation}
\Pi\dvtx\ell^2(\mathbb C) \stackrel{\psi} {\longrightarrow}
L^2(\T) \stackrel{\chi_{[0,1/2]}} {\longrightarrow}
L^2(\T) \stackrel{\psi^{-1} } {\longrightarrow}
\ell^2(\mathbb C), \label{eqPidef}
\end{equation}
where $\psi\dvtx\ell^2(\mathbb C) \to L^2(\T) $ is the Fourier transform
which sends the coordinate vector $e_m, m \in\mathbb Z$ to the
periodic function $x \mapsto e^{2\pi\ci mx} \in L^2(\T)$, $\psi^{-1}$
is the inverse map of $\psi$ and $\chi_{[0,1/2]}\dvtx L^2(\T) \to
L^2(\T) $
is the projection map given by $f \mapsto f\times 1_{[0,1/2]}$. The
operator $\Pi$ is Hermitian and is defined on the entire $\ell
^2(\mathbb C)$ and hence is self-adjoint. It is easy to check that for
any $k, l \in\mathbb Z$,
\[
\langle e_k, \Pi e_l\rangle=: \Pi(k, l) = \cases{
\displaystyle\frac12, & \quad$\mbox{if } k = l,$
\vspace*{2pt}\cr
0, & \quad$\mbox{if } k \ne l, |k- l| \mbox{ is even},$
\vspace*{2pt}\cr
\displaystyle\frac{-\ci}{\pi(k-l)}, &\quad $\mbox{if } |k- l| \mbox{ is odd}.$}
\]
Here $ \langle\cdot, \cdot\rangle$ denotes the usual inner product on
$\ell^2(\mathbb C)$.

Define $2 \to4$ operator norm of $\Pi$ as
\[
\|\Pi\|_{2 \to4}:= \sup\bigl\{ \| \Pi\mathbf{v}\|_4 \dvtx
\mathbf{v}\in
\ell^2(\mathbb C), \| \mathbf{v}\|_2 \le1 \bigr\},
\]
where for any vector $\mathbf{v}\in\ell^2(\mathbb C)$ and $p \ge1$,
$\|\mathbf{v}\|
_p$ denotes the standard $\ell^p$ norm of~$\mathbf{v}$.

%Check that $|\bP_{2n}(k, l) - \Pi(k, l)|$ converges to $0$ as $n \to
%can deduce that

We claim that
%
%
%e15 #&#
\begin{equation}
\label{eqrateofconv} \bigl|\bP_{2n}(k, l) - \Pi(k, l)\bigr| = O
\bigl(n^{-1}\bigr)\qquad \mbox{as } n \to\infty\mbox{ and } |k-l| = o(n).
\end{equation}
There is nothing to prove if $|k-l|$ is even. So assume that $|k - l| $
is odd and $|k-l| = o(n)$. In this case we can write $\bP_{2n}(k, l) -
\Pi(k, l)$ as
$ n^{-1} \times\frac{1}{x_n} [ \frac{x_n}{ 1 - e^{ - \ci x_n} }
+ \ci] $ where $x_n = \pi(k - l)/n = o(1)$. Now \eqref
{eqrateofconv} easily follows from the following limit, which is elementary:
\[
\lim_{x \to0} \frac{1}{x} \biggl[ \frac{x}{ 1 - e^{ - \ci x} } +
\ci\biggr]  = \lim_{x \to0} \frac{ 1 - \sin x/x + \ci(1-\cos
x)/x }{
(1-\cos x)+ \ci\sin x} = \frac12.
\]

We will make use of \eqref{eqrateofconv} when we prove the upper bound
for the top eigenvalue in Section~\ref{supper}.

%s3.2 #&#
\subsection{\texorpdfstring{Allowing $\eps$ room}{Allowing epsilon room}}
As one might guess, the diagonal
entries of $\bD$ which have small absolute value should not have too
much influence on determining the value of $\lambda_1(\bP\bD\bP)$. In
this subsection, we will make this idea precise. For $\eps>0$, consider
the random set
\[
S=\bigl\{0\le j\le2n-1 \dvtx|d_j|\ge\eps\sqrt{2\log n}\bigr\},
\]
let $\bR:= \operatorname{diag}(1_{j\in S})$ and let $\bD^{\eps
}:=\bD\bR$. Then
%
%
%e16 #&#
\begin{eqnarray}
\label{deps} \biggl\llvert\frac{ \lambda_1( \bP\bD\bP
)}{\sqrt{2\log n} } - \frac{\lambda_1( \bP\bD^{\eps} \bP)}
{\sqrt{2\log n} } \biggr\rrvert
&\le&\frac{ \| \bP(\bD- \bD^{\eps}) \bP\|_\spec
}{\sqrt{2\log n} }
\nonumber
\\[-8pt]
\\[-8pt]
\nonumber
&\le&\frac{ \| \bP\|_\spec\| \bD- \bD^{\eps}\|
_\spec\|\bP\|_\spec}{\sqrt{2\log n} } \le\eps.
\end{eqnarray}

%s3.3 #&#
\subsection{Random partition of the interval} \label{subsecpartition}
For a set $B$, we denote by $\#B$ the cardinality of $B$.

Set $r =\lceil\log n \rceil^3$, and let $m = \lfloor n/ 2r
\rfloor$. Divide the interval $\{0, 1, 2, \ldots, 2n-1\}$ into
$2m+1$ consecutive disjoint subintervals (called \emph{bricks})
$L_{-m}, \ldots, L_{-1}$, $L_0, L_1, \ldots, L_m $ in such a way that:
\begin{itemize}
\item$0 \in L_{-m}$ and $n \in L_0$,
\item the length of each $L_i$ is between $r$ and
$4r$ and
\item the subdivision is symmetric about $n$: $L_i=2n-L_{-i}$ for $-m <
i< m$ and $L_{-m} \setminus\{0\} = 2n - L_m$.
\end{itemize}

%
%with as below.
%and
%-j: j \in L_{-k}\} \mbox{ for all } 1 \le k < m. \]
%We call each the sets $L_k$ a \textit{brick}. Note that $r \le\#L_k \le
%4r$ for all $k$.

We define a \emph{block} to be a (nonempty) union of consecutive
bricks, say
\[
J = L_j \cup L_{j+1} \cup\cdots\cup L_{k-1} \cup
L_k,
\]
with $-m \le j \le k \le m$. We fix $\eps>0$, and set $M = M(\eps): = 4
+ 12/
\eps^{2}$. We will call a block $J$ \emph{admissible} if:
\begin{longlist}[(a)]
\item[(a)] either $J \subseteq L_{-m+1} \cup L_{-m+2} \cup\cdots\cup
L_{-1}$ or $J \subseteq L_{1} \cup L_{2} \cup\cdots\cup L_{m}$ and
\item[(b)] the number $1+k - j$ of bricks that make up the block $J$ is
at most
$M(\eps)$.
\end{longlist}
The set of all \emph{admissible blocks} will be denoted by $\lcal$. We
mention here that the notion of admissible blocks depends on the fixed
parameter $\eps$. Notice that (a) is equivalent to requiring $L_{-m}
\nsubseteq J$ and $L_0 \nsubseteq J$. Moreover, if $L \in\lcal
$, then size of $L$ is
bounded below and above by $r$ and $4Mr$, respectively.

%Let $\lcal$ be the collection of nonempty sets consisting of the
%intervals formed by taking union at most $M = M(\eps): = 4 + 12/
%L_{-m+2}, \ldots, L_{-1}$ or from $L_{1}, L_{2}, \ldots,
%L_{m-1}$. We will call an element of $\lcal$ as an {\bf
%admissible block}.

We define a brick $L_k$ to be \textit{visible} if $L \cap S \ne
\varnothing$. Otherwise $L_k$ is called \textit{invisible}. Clearly,
for $1 \le k <m$, $L_k$ is visible if and only if $L_{-k}$ is
visible. Given the random set $S$, partition $ \{0, 1, 2, \ldots, 2n -
1\}$ into disjoint intervals $J$ by subdividing between each pair of
consecutive invisible bricks. Clearly, each such $J$ is a block. Denote
the collection of all those $J$'s by $\Lambda$. Note that the $\Lambda$
and the $J$'s are at random. Note that in this
random partition, each block $J \in\Lambda$ starts and ends with a
``gap'' of
size at least $(\log n)^3$. More precisely, for any $[a, b]\in
\Lambda$, we have $[a, a + (\log n)^3] \cap S = \varnothing$
(unless $a=0$) and $[b- (\log n)^3, b] \cap S = \varnothing$ (unless
$b=2n-1$).

%pr8 #&#
\begin{proposition} \label{proppartition} For each $\eps>0$ and each
division $\lcal$,
the following holds with high probability:
For each $J\in\Lambda$, if $J \cap S \ne\varnothing$, then:
\begin{longlist}[(1)]
\item[(1)] $J \in\lcal$ and
\item[(2)] $\# ( J \cap S) \le M$ for all $J \in\mathcal L$.
\end{longlist}
\end{proposition}

\begin{pf}
For any fixed $s \ge1$ and $0 \le j_1< j_2< \cdots< j_s \le n$,
%
%
%e17 #&#
\begin{eqnarray}\label{ineqtailestimate}
&&\prob\bigl\{ |d_{j_i}| > \eps\sqrt{2\log n}, 1 \le i \le s \bigr\} \nonumber\\
&&\qquad\le\sum
_{
\beta_i \in\{-1, +1\} } \prob\Biggl\{ \sum
_{ i=1}^s \beta_i d_{j_i} >
s \eps\sqrt{2\log n} \Biggr\}
\nonumber
\\[-8pt]
\\[-8pt]
\nonumber
&&\qquad\le2^s \exp\biggl( -\frac{ \eps^2 s^2 \cdot\log n}{ \operatorname
{Var}(\sum_{ i=1}^s \beta_i d_{j_i}) + O(s n^{1/\gamma-1/2})\cdot s
\eps\sqrt{2\log n}} \biggr)
\\
&&\qquad = O\bigl(
n^{ - s\eps^2/3} \bigr).\nonumber
\end{eqnarray}
The second inequality in \eqref{ineqtailestimate} is a consequence of
Bernstein's
inequality once we write $\sum_{ i=1}^s \beta_i d_{j_i} $ as the linear
sum of $a_i$'s. Note that coefficient of each $a_j$ in the sum is of
the order of $s \times O(n^{-1/2})$, and on other hand, each $a_j$ is
bounded by $n^{1/\gamma}$. For the third inequality above we used the
fact that $
\operatorname{Var}(\sum_{ i=1}^s \beta_i d_{j_i}) = \sum_{ i=1}^s
\operatorname{Var}(d_{j_i}) \le(s+2)$ (by Lemma~\ref{lemdjmeancov}).

Note that if ${(1)}$ or ${(2)}$ fails, then one of the
following events holds:
\begin{longlist}[(iii)]
\item[(i)] either of the bricks $L_0$ and $L_{-m}$ is visible;
\item[(ii)] there exists a stretch of $M$ consecutive bricks from $L_{
-m+1},L_{-m+2}, \ldots,\break  L_{-1}$ such that at least $\lfloor
M/2\rfloor-1$ of them are visible;
\item[(iii)] there exists a stretch of $M$ consecutive bricks (say,
$L_a, L_{a+1}, \ldots,\break L_{a+M-1}$) from $L_{ -m+1}, L_{-m+2}, \ldots
, L_{-1}$ such that $\sum_{i=0}^{M-1} \#(L_{a+i} \cap S) \ge M$.
\end{longlist}
By \eqref{ineqtailestimate}, we have $\prob\{ \mbox{event (i)} \} =
O(n^{-\eps^2/3})$. We observe that events (ii) and (iii) are both
contained in the following event:
\begin{longlist}[(iii)]
\item[(iv)] there exists a stretch of $M$ consecutive bricks (say,
$L_a, L_{a+1}, \ldots,\break L_{a+M-1}$) from $L_{ -m+1}, L_{-m+2}, \ldots,
L_{-1}$ such that $\sum_{i=0}^{M-1} \#(L_{a+i} \cap S) \ge\break\lfloor
M/2\rfloor-1$.
\end{longlist}
Again by \eqref{ineqtailestimate}, if we fix a position of such $M$
consecutive blocks $L_a, L_{a+1}, \ldots,\break  L_{a+M-1}$ and then fix $s=
\lfloor M/2\rfloor-1$ positions $j_1, j_2, \ldots$ within the blocks
$L_a, L_{a+1}, \ldots, L_{a+M-1}$, the probability that $j_1, j_2,
\ldots, j_s \in S$ is bounded above by $O( n^{ - s\eps^2/3} ) = O(n^{-2})$.
Hence by union bound, $\prob\{\mbox{event (iv)}\} =\break O(n (\log
n)^{3M/2} \times n^{-2})$. This is because we can choose the index $a$ from
the set $\{ -m +1, -m+2, \ldots, -1\}$ in at most $m \le n$ ways and
$s$ positions $j_1, j_2, \ldots$ can be selected in the blocks $L_a,
L_{a+1}, \ldots, L_{a+M-1}$ in at most ${ 4Mr \choose s} \le O((\log
n)^{3M/2})$ ways.

This implies that the probability that either of the events (i) or
(iv) happens goes to $0$ as $n \to\infty$. This completes the proof.
\end{pf}

%s3.4 #&#
\subsection{Reduction to a block diagonal form}
Let $\bB$ be the following block diagonal form of the matrix $\bP$:
\[
\bB(k, l) = \cases{ %
\bP(k, l), & \quad$\mbox{if } k, l
\in J \mbox{ for some } J \in\Lambda,$
\vspace*{2pt}\cr
0, & \quad $\mbox{otherwise}.$}
\]

%le9 #&#
\begin{lemma}\label{lemblockform}
For each $\varepsilon>0$, there exists $K >0$ such that with high probability,
\[
\Bigl\llvert\lambda_1 \bigl( \bP\bD^{\eps} \bP\bigr) -
\max_{ J\in\Lambda: J
\cap S \ne\varnothing} \lambda_1 \bigl( \bP[J ]
\bD^{\eps} [J] \bP[J ] \bigr)\Bigr\rrvert\le K.
\]
\end{lemma}
\begin{pf}
The maximum above equals $\lambda_1( \bB\bD^{\eps} \bB)$, so
the left-hand side is bounded above by
\begin{eqnarray*}
\bigl\| \bP\bD^{\eps} \bP- \bB\bD^{\eps} \bB\bigr\|_\spec&\le&
\bigl\| (\bP- \bB) \bD^{\eps} \bP\bigr\|_\spec+ \bigl\| \bB\bD^{\eps}
(\bP- \bB) \bigr\|_\spec
\\
&\le&\bigl\| (\bP- \bB) \bD^{\eps}\bigr\|_\spec\|\bP\|_\spec+
\| \bB\| _\spec\bigl\| \bD^{\eps} (\bP- \bB) \bigr\|_\spec.
\end{eqnarray*}
Note that since $\bP$ is a projection matrix $\| \bP\|_\spec= 1$
and by block-diagonality we have $\| \bB\|_\spec= \max_{J \in
\Lambda} \| \bP[J] \|_\spec$. The matrix $\bP[J]$ is just $\bP$
conjugated by a coordinate projection, so it has norm at most 1.

Since $\bD^\eps=\bR\bD$ we first bound the spectral norm of $(\bP
- \bB) \bR$. The maximal column sum of this matrix is bounded
above by the maximal absolute row sum of~$\bP$, which by \eqref{eqProwsum}
is at most $C_2 \log n$. Note that $((\bP-\bB)\bR)(k,l)=0$ unless
$k,l$ are in different parts of $\Lambda$. This gives the upper
bound for the maximal absolute row sum
\[
\max_{J \in\Lambda, k \in J } \sum_{ l \notin J} \bR(l, l)
\bigl|\bP(k, l) \bigr|
\le C\sum_{ k =1}^{\# \Lambda-1} \frac{2M}{k (\log n)^3} =
O\bigl((\log n)^{-2}\bigr),
\]
which holds with high probability. We used the fact that with high
probability each part in $\Lambda$ has at most $M$ elements $j$
where $\bR(j,j)$ is nonzero (Proposition~\ref{proppartition}),
and different parts have gaps of size $(\log n)^3$ in between
them. Since the spectral norm is bounded above by the geometric
mean of the maximal row and column sums, we get
$\|(\bP-\bB)\bR\|_\spec= O((\log n)^{-1/2})$.

%
%and the similar bound holds also for $\| \bD^{\eps} (\bP- \bB)
%k=0}^{2n-1} |\bP(k, l)| \le C_2 \log n.
% \]
%As a consequence of Proposition~\ref{proppartition}, the
%following bound holds \bwhp,
% \max_k \sum_{ l=0}^{2n-1} \bR(l, l) |(\bP- \bB)(k, l)| & \le\max_{J
% &\le\sum_{ k =1}^{\# \Lambda-1} \frac{2M}{k (\log n)^3} = O((\log
%n)^{-2}). \end{eqnarray*}
%Hence \bwhp, $\| (\bP- \bB) \bD^{\eps}\|_\spec^2 \le\max_j
%|d_j|^2 \cdot O((\log n)^{-1}).$
%
By Bernstein's inequality, we can find a constant $C_3$ such that
with high probability $\|\bD\|_\spec= \max_j |d_j|$ is at most
$C_3 \sqrt{ \log n}$. Therefore, $\| (\bP- \bB)\bD^{\eps}
\|_\spec^2 = O(1)$ with high probability. Similarly, $\| \bD^{\eps}
(\bP-
\bB) \|_\spec^2 = O(1)$ with high probability, which yields the
desired result.
\end{pf}

%s4 #&#
\section{Proof of the upper bound}\label{supper}
Assume that $a_0, a_1, \ldots, a_n$ are independent mean zero, variance
one random variables which are uniformly bounded by $n^{1/\gamma}$.
This section consists of the proof of the upper bound.
%
%pr10 #&#
\begin{proposition}\label{proupperbound} For each $\varepsilon>0$, there
exists $c_n=o(1)$ such that with high probability,
\[
\max_{ J\in\Lambda: J\cap S \ne\varnothing} \frac{\lambda_1( \bP
[J ]
\bD^{\eps} [J] \bP[J ])} {\sqrt{2\log n} } \le\| \Pi\|^2_{2 \to4}
+ c_n.
\]
\end{proposition}

%le11 #&#
\begin{lemma} \label{uppertailbd}
Fix $k \ge1$ and $\delta_1, \delta_2, \ldots, \delta_k \ge0$. Let
$\|
\delta\|_2:= (\delta_1^2 + \cdots+ \delta_k^2)^{1/2}$. Then for
sufficiently large $n$ and for any $0 < j_1 < j_2 < \cdots< j_k < n$,
\[
\prob\bigl\{ |d_{j_1}| > \delta_1 \sqrt{ 2 \log n},
|d_{j_2}| > \delta_2 \sqrt{ 2 \log n}, \ldots,
|d_{j_k}| > \delta_k \sqrt{ 2 \log n}\bigr\} \le
2^{k+1}n^{-\| \delta\|_2^2 }.
\]
\end{lemma}
\begin{pf}
To avoid triviality, assume that $\| \delta\|_2 >0$. Clearly,
\begin{eqnarray*}
&&\prob\bigl\{ |d_{j_1}| > \delta_1 \sqrt{ 2 \log n},
\ldots, |d_{j_k}| > \delta_k \sqrt{ 2 \log n} \bigr\}
\\
&&\qquad \le\sum_{ \beta_i \in\{ \pm
1\} }\prob\Biggl\{ \Biggl( \sum
_{i=1}^k \beta_i
\delta_i d_{j_i} \Biggr) \Big/ \| \delta\|_2 > \|
\delta\|_2 \sqrt{ 2 \log n} \Biggr\}.
\end{eqnarray*}
By Lemma~\ref{lemdjmeancov}, for any choice of $\beta_i \in\{
\pm1\}$, the sum $ (\beta_1\delta_1 d_{j_1} + \cdots+
\beta_k\delta_k d_{j_k} )/\break \| \delta\|_2$ has variance one and
can be expressed as a linear combination independent random
variables as $ \sum_{i=0}^n \theta_i a_i$ for suitable real
coefficients $\theta_i$ with $|\theta_i| \le2 k n^{-1/2}$ and $
\sum_{i=0}^n\theta_i^2 =1$. Recall that $|a_i| \le n^{1/\gamma}$
so that we can apply Bernstein's inequality to obtain
\begin{eqnarray*}
&&\prob\Biggl\{ \Biggl( \sum_{i=1}^k
\beta_i \delta_i d_{j_i} \Biggr)\Big / \| \delta
\|_2 > \| \delta\|_2 \sqrt{ 2 \log n} \Biggr\} \\
&&\qquad\le\exp
\biggl( - \frac{\| \delta\|_2^2 \log n}{1 + \| \delta\|_2 \sqrt{ 2
\log n} \cdot2k n^{1/\gamma- 1/2}/3 } \biggr)
\\
&&\qquad \le2 n^{-\|\delta\|^2_2}
\end{eqnarray*}
for sufficiently large $n$. This completes the proof of the lemma.
\end{pf}

%le12 #&#
\begin{lemma}\label{lemuppertailbd2}
Fix $\eta>0$, $k \ge1$. Then there exists a constant $C_4 = C_4(\eta,
k)$ such that for sufficiently large $n$ and for any $0 < j_1 < j_2 <
\cdots< j_k < n$, we have
%
%
%e18 #&#
\begin{eqnarray*}
&&\prob\bigl\{ |d_{j_1}| > \delta_1 \sqrt{ 2 \log n},
|d_{j_2}| > \delta_2 \sqrt{ 2 \log n}, \ldots,
|d_{j_k}| > \delta_k \sqrt{ 2 \log n}\\
&&\hspace*{14pt}\mbox{for some }\delta_1, \ldots, \delta_k > 0 \mbox{ such that }
\delta_1^2 + \delta_2^2 + \cdots+
\delta_k^2 \ge1+\eta\bigr\}\\
&&\qquad \le C_4n^{-(1+\eta/2) }.
\end{eqnarray*}
\end{lemma}
\begin{pf}
Construct an $\frac{\eta}{4(1+\eta)k}$-net $\mathcal N$ for the
interval $[0, 1 +\eta]$ by choosing $ \lfloor\frac{4(1+\eta
)^2k}{\eta}
\rfloor+1$ equally spaced points in $[0, 1+\eta]$ including both
endpoints $0$ and $1+\eta$. Therefore,
given any $\delta_1, \ldots, \delta_k \in[0, 1+\eta]$, we can find
$\alpha_1, \alpha_2, \ldots, \alpha_k \in\mathcal N$ such that
$ \delta_i - \frac{\eta}{4(1+\eta)k} \le\alpha_i \le\delta_i$ for
each $i$. This implies that $\alpha_1^2 +\alpha_2^2 + \cdots+ \alpha
_k^2 > \sum_{i=1}^k ( \delta^2_i - 2 \delta_i \frac{\eta}{4(1+\eta
)k} )\ge\delta_1^2 + \delta_2^2 + \cdots+ \delta_k^2 - \eta/2 $.
Clearly, the event $\{ |d_{j_1}| > \delta_1 \sqrt{ 2 \log n}, \ldots,
|d_{j_k}| > \delta_k \sqrt{ 2 \log n}$ for some
$\delta_1, \ldots, \delta_k \in[0, 1+\eta]$ such that $\delta_1^2 +
\delta_2^2 + \cdots+ \delta_k^2 \ge1+\eta\}$ is contained in the
finite union of events $\{ |d_{j_1}| > \alpha_1 \sqrt{ 2 \log n},
\ldots
, |d_{j_k}| > \alpha_k \sqrt{ 2 \log n} \}$ where the union is taken
over for all possible choices of $ \alpha_i \in\mathcal N$ for all $i$
such that $\alpha_1^2 +\alpha_2^2 + \cdots+ \alpha_k^2 \ge1+ \eta/2$.
Now we use the union bound and Lemma~\ref{uppertailbd} to conclude the
probability of the given event is bounded by
$2^{k+1}(\# \mathcal N )^k n^{-(1+\eta/2)}$. On the other hand, it
again follows from Lemma~\ref{uppertailbd} that the probability of the
event $\{ |d_{j_i}| > (1+\eta) \sqrt{ 2 \log n} \mbox{ for some } 1
\le
i \le k \}$ is bounded by $2k n^{-(1+\eta)^2} \le2k n^{-(1+\eta/2)}$.

We complete the proof by taking $C_4 = 2^{k+1}(\# \mathcal N )^k + 2k$.
\end{pf}
%
%Fix $0< \eta< 1$ and $k \in\mathbb N$. Let $\xi= (\xi_1, \xi_2,
%are i.i.d.\ standard Gaussian random variables. Then there exists a
%constant $B$ such that for all $n \ge2$, we have
%Note that $\| \xi\|_2^2$ has standard $\chi^2_k$ distribution whose
%pdf has the form \[f(x) = c x^{k/2-1} e^{-x/2}\1_{ \{x>0\}},\]
%where $c$ is the normalisation constant. Therefore, we have
% &= c \int_{2(1+\eta)^2 \log n}^\infty\exp\left({- \frac{x}{2} \Big
%(1 - \frac{(k-2)\log x}{x} \Big)} \right)dx\\
% &\le c \int_{2(1+\eta)^2 \log n}^\infty\exp\left(- \frac{x}{2} \Big
%(1 - \frac{(k-2)\log\log n}{\log n} \Big) \right)dx\\
% &\le4c \exp\left(- \frac{ 2(1+\eta)^2 \log n}{2} \Big(1 -
% &\le4c (\log n)^{ 4(k-2) }n^{-(1+\eta)^2}.
%Now we take $B=4c$ to conclude the proof.

%co13 #&#
\begin{corollary}
Let $\eps>0$, and let $M = M(\eps)$ be as defined in Section~\ref{subsecpartition}.
For every $\eta>0$, with high probability, for all
admissible
blocks $L\in\mathcal L$ and all distinct $j_1,\ldots,j_M\in L$ we
have $d_{j_1}^2+\cdots+d_{j_M}^2\le(1+\eta) \sqrt{2 \log n}$.
\end{corollary}

\begin{pf}
For a fixed admissible block and points $j_i$ the probability that
the claim is violated is at most $C_4n^{-(1+ \eta/2) }$ by the
lemma. By union bound, the probability that the claim is violated
is at most $ n { 4M(\log n)^3 \choose M} \cdot C_4 n^{-(1+\eta/2)} =
O((\log n)^{3M}n^{-\eta/2})$. This is because there can be at most $n$
admissible blocks, the length of an admissible block can be at most $4M
(\log n)^3 $ and the number of ways $M$ distinct indices can be chosen
from an admissible block is at most ${ 4M(\log n)^3 \choose M}$.
\end{pf}

By part (2) of Proposition~\ref{proppartition}, with high probability,
$\bD^\eps$ contains at most $M$ nonzero entries in every
admissible block. Thus it follows that for each $\eta>0$ with high
probability for all $L\in\mathcal L$,
\[
\sum_{j\in L} \bigl(\bD^{\eps}_{j,j}
\bigr)^2 \le1+\eta.
\]
Therefore, we have
%
%
%By Proposition~\ref{proppartition} \bwhp, for any $J\in
%J\cap S$ has size at most $M$.
% \max_{ J \in\Lambda: J \cap S \ne\varnothing} \frac{\lambda_1( \bP[J
%] \bD^{\eps} [J] \bP[J ])} {\sqrt{2\log n} } \le\max_{L \in\mathcal
%L, \bB\in\bcal} \frac{\lambda_1( \bP[L ] \bD^{\eps} [L] \bB\bP[L
%])} {\sqrt{2\log n} }
% \end{equation}
%where $\bcal= \bcal(\#L) $ is the set of all $\# L \times\# L$
%nonzero diagonal matrices that have all zeros on the diagonal except
%for at most $M$ ones.
%
%Let $\eta>0$ be a small constant. Fix an admissible block $L \in
%for some $1 \le s \le M$. Then by Lemma~\ref{lemuppertailbd2},
%the probability of the event that
%n} \mbox{ \ \ for some } \\
% \end{split}\]
%is bounded by $\le C_4n^{-(1+ \eta/2) }$. So, by union bound, the
%probability that the above event happens for \textit{at least one}
%admissible block $L \in\lcal$ and for \textit{any} $s$ distinct
%points in $L$ for some $1 \le s \le M$ is bounded above by $ n {
%M(\log n)^3 \choose M} \cdot C n^{-(1+\eta/2)} = O((\log
%n)^{3M}n^{-\eta/2})$.
%Combined with \eqref{eqrandomtofixedblock},
%this implies that \bwhp
%
%
%e19 #&#
\begin{eqnarray}\label{eqoptimzedeltabound}
\qquad&&\max_{ J \in\Lambda: J \cap S \ne\varnothing} \frac{\lambda_1 (
\bP
[J ] \bD^{\eps} [J] \bP[J ] )} {\sqrt{2\log n} }
\nonumber
\\[-8pt]
\\[-8pt]
\nonumber
%&\le\sup\Big\{ \lambda_1 \Big( \bP[L ] \operatorname{Diag}(\delta_1,
%L, \bB\in\bcal, \sum_{j=1}^{\#L} \bB_{jj} \delta_j^2 \le1+ \eta\Big
&&\qquad\le\sup\Biggl\{
\lambda_1 \bigl( \bP[L ] \operatorname{Diag}(\delta_1,
\delta_2, \ldots, \delta_{ \#L}) \bP[L ] \bigr) \dvtx L \in
\mathcal L, \sum_{j=1}^{\#L}
\delta_j ^2 \le1+ \eta\Biggr\},
\end{eqnarray}
which holds with high probability due to part (1) of Proposition~\ref
{proppartition}.
Now by \eqref{eqrateofconv}, we have $\max_{ L \in\lcal} \|
\bP[L] - \Pi[L] \|_{\spec} = O ( \frac{(\log
n)^3}{n} )$. Therefore, if we take $q=q(n):= 4M \lceil\log
n \rceil^3$ and $H = [1, q] \cap\mathbb Z$, then
\eqref{eqoptimzedeltabound} can be bounded by
\begin{eqnarray*}
\eqref{eqoptimzedeltabound} &\le&\sup\Biggl\{ \lambda_1 \bigl(
\Pi[L ] \operatorname{Diag}(\delta_1, \delta_2, \ldots,
\delta_{ \#L}) \Pi[L ] \bigr) \dvtx L \in\mathcal L, \sum
_{j=0}^{\#L} \delta_j ^2 \le1+
\eta\Biggr\}
\nonumber\\
&&{} + O \biggl( \frac{(\log n)^3}{n} \biggr)
\nonumber
\\[-8pt]
\\[-8pt]
\nonumber
&\le&(1+\eta)^{1/2}\sup\Biggl\{ \lambda_1 \bigl( \Pi[ H]
\operatorname{Diag}(\delta_1, \delta_2, \ldots,
\delta_{ q}) \Pi[H ] \bigr) \dvtx\sum_{j=1}^{q}
\delta_j ^2 \le1 \Biggr\}
\\
&&{} + O \biggl( \frac{(\log n)^3}{n} \biggr).\nonumber
\end{eqnarray*}
By Rayleigh's characterization of the maximum eigenvalue of
Hermitian matrices, the above supremum equals
\[
\sup\Biggl\{ \bigl\langle\mathbf{v}, \Pi[ H] \operatorname
{Diag}(\delta_1,
\delta_2, \ldots, \delta_{ q}) \Pi[H ] \mathbf{v}\bigr\rangle
\dvtx
\sum_{j=1}^{q} \delta_j
^2 \le1, \mathbf{v}\in\mathbb C^q, \|\mathbf{v}\|_2 \le1
\Biggr\}.
\]
Denote by $\delta$ the infinite dimensional vector $(\ldots,
\delta_{-1}, \delta_0, \delta_1, \ldots)$ in $\ell^2(\mathbb R)$.
Now extending the range of optimization we get the upper bound
%
%
%e20 #&#
\begin{equation}
\label{Piopt} \sup\bigl\{ \bigl\langle\mathbf{v}, \Pi
\operatorname{Diag}(\delta) \Pi\mathbf{v}
\bigr\rangle\dvtx\delta\in\ell^2(\mathbb R), \|\delta\|_2
\le1, \mathbf{v}\in\ell^2(\mathbb C), \|\mathbf{v}\|_2 \le1 \bigr
\}.
\end{equation}
Now note that (with $\odot$ denoting coordinate-wise
multiplication)
\[
\bigl\langle\mathbf{v}, \Pi\operatorname{Diag}(\delta) \Pi
\mathbf{v}\bigr\rangle= \bigl\langle
\Pi\mathbf{v}, \operatorname{Diag}(\delta) \Pi\mathbf{v}\bigr
\rangle= \langle\delta, \overline{\Pi
\mathbf{v}} \odot\Pi\mathbf{v}\rangle,
\]
so if we fix $\mathbf{v}$, this is maximized when $\delta$ equals
$\overline{\Pi\mathbf{v}} \odot\Pi\mathbf{v}$ divided by its
length. The
maximum, for $\mathbf{v}$ fixed is $ \|\overline{\Pi\mathbf
{v}}\odot\Pi
\mathbf{v}\|_2=\|\Pi\mathbf{v}\|_4^2 $, so expression \eqref{Piopt} equals
$\| \Pi\|_{2 \to4}^2$. This completes the proof of Proposition~\ref
{proupperbound}.

%s5 #&#
\section{Proof of the lower bound}\label{slower}
Assume that $a_0, a_1, \ldots, a_n$ are independent mean zero, variance
one random variables which are uniformly bounded by $n^{1/\gamma}$.
This section consists of the proof of the lower bound.
%
%pr14 #&#
\begin{proposition}\label{prolowerbound} For each $\tau>0$ and $\eps
>0$, with high probability,
\[
\max_{ J\in\Lambda\dvtx J\cap S \ne\varnothing} \frac{\lambda
_1( \bP[J ]
\bD^{\eps} [J] \bP[J ])} {\sqrt{2\log n} } \ge\| \Pi\|^2_{2 \to4}
- 2\tau.
\]
\end{proposition}
Let $G_0, G_1, \ldots, G_n$ be i.i.d. standard Gaussians independent
of $(a_i)_{ 0 \le i \le n}$. Recall
\[
d_j(a) = \frac{1}{\sqrt{2n}} \Biggl[ \sqrt{2} a_0 +
(-1)^j\sqrt{2} a_n + 2\sum_{k=1}^{n-1}
a_k \cos\biggl( \frac{2 \pi j k }{n} \biggr) \Biggr], \qquad 0 \le j < 2n.
\]
We will use $d_j(G)$ to refer to the above sum with $a_i$ being
replaced by $G_i$ for each $0 \le i \le n$.
Suppose we are given some \textit{nonzero} real numbers $u_1, u_2,
\ldots,
u_k$ for some fixed $k \ge1$ such that $u_1^2 + u_2^2+ \cdots+ u_k^2
< 1$. Fix any $\eta>0 $ small such that $\eta< |u_i|$ for each $i$ and
if we set $u_s':=|u_s|-\eta>0$ for $ 1\le s \le k$, then $(u_1')^2+
\cdots+ (u_k')^2 <1$. Take $p = p(n):= 100\lceil\log n\rceil^3, b =
b(n):= 12\lceil\log n \rceil^3$ and $N= N(n):= \lfloor n/ 2p
\rfloor
$. For each $1 \le j \le N$, define the intervals
\[
I_j = %
\cases{(-\eps,\eps), &\quad $\mbox{for } -b+1\le j\le0,$
\vspace*{2pt}
\cr
(u_j-\eta,u_j+\eta), &\quad $\mbox{for } 1\le j
\le k,$\vspace*{2pt}
\cr
(-\eps,\eps), &\quad $\mbox{for } k+1\le j\le k+b,$}
\]
and let $A_i$ be the event that $d_{ip+j}\in\sqrt{2\log n} I_j$ for
all $j=-b+1,\ldots, b+k$.

%$d_{ip+1}(a) \in\Big( (u_1 -\eta)\sqrt{2 \log n}, (u_1+\eta)\sqrt{2
%n}, (u_k+\eta)\sqrt{2 \log n}\Big)$ and $|d_j(a)| \le\eps\sqrt{2
%ip-2, \ldots, ip-b$.

%pr15 #&#
\begin{proposition} \label{propsecondmoment}
Let $A_i $ be as above. Then as $n \to\infty$, the probability that at
least one of the events $A_1, A_2, \ldots, A_N$ happens converges to one.
\end{proposition}

\begin{pf}
The proof is based on the second moment method. First of all, fix any
smooth function $\psi\dvtx\mathbb R \to[0, 1]$ such that $\psi( x)
= 0 $
for $x \le0$ and $\psi(x) = 1$ for $x \in[1, \infty)$.
For any reals $a<b$, the indicator function of the interval $(a \sqrt{2
\log n}, b\sqrt{2 \log n}) $ can be bounded below by the smooth function
\[
\1_{ (a \sqrt{2 \log n}, b\sqrt{2 \log n})}(x) \ge\zeta_{(a,
b)}(x):= \psi( x -a\sqrt{2 \log n} )
\psi( b\sqrt{2 \log n} -x ).
\]
For a fixed $a<b$, the first, second and third derivatives of the
function $\zeta_{(a, b)}$ are all bounded by some constant in the
supremum norm.
Define, for each $i$ in $1 \le i \le N$,
%
%
%e21 #&#
\begin{equation}
\label{defWi} W_i = \prod
_{ j=-b+1}^{b+k} \zeta_{I_j}(d_{ip+j}).
\end{equation}
Clearly, for each $i$, $ W_i \le\1_{A_i}$. Thus by the Paley--Zygmund
inequality,
\[
\prob\Biggl\{ \sum_{ i=1}^N
\1_{ A_i} \ge1\Biggr\} \ge\prob\Biggl\{ \sum
_{ i=1}^N W_i >0\Biggr\} \ge
\frac{\E[(\sum_{ i=1}^N W_i)^2 ] }{(\E\sum_{ i=1}^N W_i )^2}.
\]
So, the proposition follows if we can show
\[
\E\Biggl[\Biggl(\sum_{ i=1}^N
W_i\Biggr)^2 \Biggr] = \bigl(1+o(1)\bigr) \Biggl(\E\sum
_{ i=1}^N W_i
\Biggr)^2.
\]
The rest of the proof is devoted to establishing this.

We will write $W_i^G$ to denote the random variable corresponding to
\eqref{defWi} with $d_j(G)$ in place $d_j(a)$ for all $j$. Recall
that $(d_j(G))_{ 0< j< n}$ is a sequence of i.i.d. standard Gaussian
random variables. Using the following well-known Gaussian tail estimates,
\[
\frac{x}{x^2+1} \frac{e^{-x^2/2}}{\sqrt{2 \pi} } \le\prob\{ G_0 >
x\} \le
\frac{1}{x} \frac{e^{-x^2/2}}{\sqrt{2 \pi} },\qquad x>0,
\]
we obtain, for all large $n$,
%
%
%e22 #&#
\begin{eqnarray}\label{gassianEestimate}
\E W_i^G &= &\prod_{j=- b+1}^{b+k}
\E\zeta_{I_j}\bigl(d_{ip+j}(G)\bigr)\nonumber\\
& =& \bigl(1 - O
\bigl(n^{-\eps^2}\bigr)\bigr)^{2b } \prod
_{ s=1}^k \Omega\biggl( \frac{1}{
\sqrt{\log n} \cdot n^{u_s'^2}} \biggr)
\\
&=& \Omega\bigl( (\log n)^{-k/2} \cdot n^{-((u'_1)^2+ \cdots+
(u'_k)^2)} \bigr),\nonumber
\end{eqnarray}
uniformly over $1 \le i \le N$, which implies that $ \sum_{i=1}^N\E
W_i^G \to\infty$ as $n \to\infty$.
%Similarly, $ \E[ (W_i^G)^2] = O \left( n^{-((u'_1)^2+ \ldots+
%(u'_k)^2)} \right)$.
Since the random variables $W_i^G$ are bounded by $1$, we have
%
%
%e23 #&#
\begin{eqnarray}\label{eqnegvar}
\sum_{i=1}^N\operatorname{Var}
W_i^G& \le&\sum_{i=1}^N
\E\bigl[ \bigl(W_i^G\bigr)^2\bigr] \le\sum
_{i=1}^N\E W_i^G
\nonumber
\\[-8pt]
\\[-8pt]
\nonumber
&=& o \Biggl( \Biggl(\sum_{i=1}^N\E
W_i^G \Biggr)^2 \Biggr).
\end{eqnarray}
For any $1 \le i \ne i'\le N$, the two subsets of indices $ \{ ip-b+1,
\ldots, ip +\break k+b \}$ and $ \{ i'p-b+1, \ldots, i'p + k+b \}$ are
disjoint. Hence, $W_i^G$ and $W_{i'}^G$ are independent of each other
and therefore,
%
%
%e24 #&#
\begin{equation}
\E\bigl[W_i^G W_{i'}^G\bigr]
= \E\bigl[W_i^G\bigr] \E\bigl[ W_{i'}^G
\bigr]. \label{gassianCovestimate}
\end{equation}
%
%= \Omega\left( (\log n)^{-k} \cdot n^{-2((u'_1)^2+ \ldots+ (u'_k)^2)}
Now \eqref{eqnegvar} and \eqref{gassianCovestimate} yield
%
%
%e25 #&#
\begin{equation}
\label{Gaussiansecondmoment} \E\Biggl[\Biggl(\sum
_{ i=1}^N W_i^G
\Biggr)^2 \Biggr] = \bigl(1+o(1)\bigr) \Biggl(\E\sum
_{ i=1}^N W_i^G
\Biggr)^2.
\end{equation}

Our next goal is to show that the differences $ \E W_i - \E W_i^G $
and $\E[W_i W_{i'}] - \E[W_i^G W_{i'}^G]$ are of smaller order for
each $i \ne i' $ using the invariance principle given in Lemma~\ref
{leminvariance}.

We apply Lemma~\ref{leminvariance} with $r = n+1, m= k+2b, \bX= (a_0,
\ldots, a_n)$, $\bY= (G_0, \ldots, G_n), f= ( d_{ip-b+1}, \ldots, d_{
ip + k+b} ) $ and
\[
g(\mathbf{z}) = \prod_{ j=-b+1}^{b+k}
\zeta_{I_j}(z_{b+j}),
\]
where $\mathbf{z}= (z_1, z_2, \ldots, z_{2b+k} )$ to obtain
%
%
%e26 #&#
\begin{equation}
\label{invremainder} \bigl| \E[W_i ] - \E
\bigl[W_i^G\bigr] \bigr| \le\sum
_{j=1}^r \E[ R_j ] + \sum
_{j=1}^r \E[T_j],
\end{equation}
where $R_j$ and $T_j$ are as defined in Lemma~\ref{leminvariance} with
%
%
%e27 #&#
\begin{eqnarray}
\label{eqhj} h_j(\mathbf{x}) &= &\sum
_{ \ell, p, q = 1}^m \partial_\ell\,
\partial_p \,\partial_q g\bigl(f(\mathbf{x})\bigr)
\,\partial_j f_\ell(\mathbf{x}) \,\partial_j
f_p(\mathbf{x})\, \partial_j f_q(\mathbf{x})
\nonumber\\
&&{}+ 3 \sum_{ \ell, p= 1}^m \partial_\ell\,
\partial_p g\bigl(f(\mathbf{x})\bigr) \,\partial^2_j
f_\ell(\mathbf{x})\, \partial_j f_p(\mathbf{x})\\
&&{} + \sum
_{ \ell= 1}^m \partial_\ell g
\bigl(f(\mathbf{x})\bigr) \,\partial^3_j f_\ell(\mathbf{x}).
\nonumber
\end{eqnarray}
We will first find an upper bound for $\E[R_j]$. The bound for $\E[
T_j]$ will be similar. Note that $\| \partial_\ell f_{t} \|_\infty\le
2 n^{-1/2}$ for each $\ell$ and $t$. Since each function $f_t$ is
linear, its higher derivatives all vanish and hence the second and the
third term of $h_j$ in \eqref{eqhj} disappear. Also, $ \| \partial
_\ell\,\partial_{\ell'} \,\partial_{\ell''} g\|_\infty= O(1)$ for each
$\ell, \ell'$ and $\ell''$ and $\partial_\ell\,\partial_{\ell'}\,
\partial
_{\ell''} g(\mathbf{z}) \ne0$ only if
$ z_{b+j} \in\sqrt{2 \log n} I_j$ for $j=1,\ldots, k$. For $x \in
\mathbb R$, define a random vector $\bZ^{(j)}(x):= ( a_0, \ldots,
a_{j-2}, x, G_{j}, \ldots, G_n)$. Therefore, we have
%
%
%e28 #&#
\begin{eqnarray} \label{eqremiander}
&&\E R_j \le\frac{c m^3}{n^{3/2}} \E\Bigl[ |a_{j-1}|^3
\sup_{ x:|x|\le|a_{j-1}| } \mathbf{1}\bigl\{ d_{ip+s}\bigl(
\bZ^{(j)}(x)\bigr)
\nonumber
\\[-8pt]
\\[-8pt]
\nonumber
&&\hspace*{97pt}\qquad \in\sqrt{2 \log n} I_s, 1 \le s \le k
\bigr\} \Bigr].
\end{eqnarray}
Note that $d_j$ depends on $a_{j-1}$ linearly with absolute coefficient
at most $\sqrt{2/n}$, and that the random variable $a_{j-1}$ is bounded
by $n^{1/\gamma}$. Thus the
supremum above is at most
\[
\mathbf{1} \bigl\{\bigl | d_{ip+s}\bigl(\bZ^{(j)}(0)\bigr)\bigr| \ge
\bigl(|u_s|-\eta\bigr) \sqrt{2 \log n} - \sqrt{2}n^{1/\gamma-1/2}, 1 \le s
\le
k \bigr\},
\]
which is independent of $a_{j-1}$.
Since $\E a_{j-1}^2 =1$, we have, with $u'_s=|u_s|-\eta$
%
%
%e29 #&#
\begin{eqnarray}\label{eqremainder2}
&&\eqref{eqremiander} \le\frac{cm^3}{n^{3/2-1/\gamma}}\prob\bigl\{
\bigl| d_{ip+s}\bigl(
\bZ^{(j)}(0)\bigr)\bigr| \ge u'_s \sqrt{2 \log n} -
\sqrt{2}n^{1/\gamma-1/2},
\nonumber
\\[-8pt]
\\[-8pt]
\nonumber
&&\hspace*{222pt}\qquad 1 \le s \le k \bigr\}.
\end{eqnarray}
Note that if we truncate the random variables $G_j, \ldots, G_n$ at
level $n^{1/\gamma}$, then we can bound \eqref{eqremainder2} using
Bernstein's inequality exactly as we did in proving Lemma~\ref
{uppertailbd}. Toward this end, we define
\[
\hat\bZ^{(j)} = ( a_0, \ldots, a_{j-2}, 0,
G_{j} \1_{\{ |G_{j}| \le
n^{1/\gamma} \}}, \ldots, G_n\1_{\{ |G_{n}| \le n^{1/\gamma} \}}).
\]
Then
%
%
%e30 #&#
\begin{eqnarray}\label{eqRestimate}
\eqref{eqremainder2} &\le&\frac{cm^3}{n^{3/2-1/\gamma}} \bigl[
\prob\bigl\{ \bigl|
d_{ip+s}\bigl(\hat\bZ^{(j)}\bigr)\bigr| \ge u'_s
\sqrt{2 \log n} - \sqrt{2}n^{1/\gamma-1/2}, 1 \le s \le k \bigr\}
\nonumber
\\
&&\hspace*{161pt}{}+ \prob\bigl\{ |G_\ell| > n^{1/\gamma} \mbox{ for some }
\ell\bigr\} \bigr]
\\
&\le&\frac{cm^3}{n^{1/2-1/\gamma}} \bigl[ O\bigl( n^{-((u'_1)^2+
\cdots+
(u'_k)^2)}\bigr) + O \bigl( n \exp
\bigl( - n^{2/\gamma}/2 \bigr) \bigr) \bigr].\nonumber
\end{eqnarray}
Hence, by combining \eqref{eqremiander}, \eqref{eqremainder2} and
\eqref{eqRestimate}, we obtain
\[
\sum_{j=1}^r \E[ R_j ] \le
O\bigl((\log n)^9 n^{1/\gamma-1/2} \cdot n^{-((u'_1)^2+ \cdots+
(u'_k)^2)}\bigr),
\]
where the constant hidden inside the big-$O$ notation above does not
depend on $i$. Similar computation yields the same asymptotic bound for
$\sum_{j=1}^r \E[ T_j ] $. Therefore, by \eqref{invremainder} and
\eqref{gassianEestimate}, we have
%
%
%e31 #&#
\begin{eqnarray}
\label{eqEdiffneg} \E[W_i ] &=& \E
\bigl[W_i^G \bigr] + O\bigl((\log n)^9
n^{1/\gamma-1/2} \cdot n^{-((u'_1)^2+ \cdots+ (u'_k)^2)}\bigr)
\nonumber
\\[-8pt]
\\[-8pt]
\nonumber
&=&\bigl(1+o(1)\bigr) \E
\bigl[W_i^G\bigr],
\end{eqnarray}
which hold uniformly in $1 \le i \le N$.
By similar argument as above, we can also show that
\enlargethispage{3pt}
%
%e32 #&#
\begin{eqnarray}
\label{eqCovdiffneg} \E[W_i
W_{i'} ] &=& \E\bigl[W_i^G
W_{i'}^G \bigr] + O\bigl((\log n)^9
n^{1/\gamma
-1/2} \cdot n^{-2((u'_1)^2+ \cdots+ (u'_k))^2}\bigr)
\nonumber
\\[-8pt]
\\[-8pt]
\nonumber
&=&\bigl(1+o(1)\bigr) \E\bigl[W_i^G
W_{i'}^G\bigr],
\end{eqnarray}
uniformly in $1 \le i \ne i' \le N$. From \eqref{eqEdiffneg} arguing
similarly as we did in \eqref{eqnegvar}, we deduce
%
%
%e33 #&#
\begin{equation}
\label{eqvarneggeneral} \sum
_{i=1}^N\operatorname{Var} W_i \le o \Biggl(\Biggl(\sum_{i=1}^N\E
W_i \Biggr)^2 \Biggr).
\end{equation}
Finally, combining \eqref{gassianCovestimate}, \eqref
{eqCovdiffneg} and \eqref{eqvarneggeneral} together, we have
\[
\sum_{ i, i' = 1}^N \E[W_i
W_{i'}] = \bigl( 1+ o(1)\bigr) \Biggl(\sum
_{ i =1}^N \E[W_i]
\Biggr)^2.
\]
%
%&= ( 1+ o(1)) \left(\sum_{ i =1}^N \E[W_i^G] \right)^2 [
%&= ( 1+ o(1)) \left(\sum_{ i =1}^N \E[W_i] \right)^2 [\mbox{by
This implies that $\prob\{ \sum_{ i=1}^N \1_{A_i} \ge1 \} = 1
- o(1)$ which completes the proof.~%
\end{pf}

For any finite $k \ge1$, we write $\Pi_k$ as a shorthand for $k
\times
k$ matrix $\Pi[\{1,2,\break \ldots,  k\}]$. Arguing along the line of the
proof of the fact $ \|\Pi\|_{2 \to4}^2 = \sup\{ \langle\mathbf{v},\break
\Pi
\operatorname{diag}(\delta) \Pi\mathbf{v}\rangle\dvtx\delta\in
\ell^2(\mathbb R), \|
\delta\|_2 \le1, \mathbf{v}\in\ell^2(\mathbb C), \|\mathbf{v}\|_2
\le1\}$, we
can also show that
%
%
%e34 #&#
\begin{equation}
\label{eqPiknorm} \|\Pi_k
\|_{2 \to4}^2 = \sup\bigl\{ \lambda_1\bigl(
\Pi_k \operatorname{diag}(\delta) \Pi_k\bigr) \dvtx\delta\in
\mathbb
R^k, \|\delta\|_2 \le1\bigr\}.
\end{equation}
Next we prove
%
%le16 #&#
\begin{lemma} \label{Lemoperatornormconv}
$\lim_{k \to\infty} \|\Pi_k\|_{2 \to4} = \|\Pi\|_{2 \to4} $.
\end{lemma}

\begin{pf} Since the operators $\psi, \psi^{-1}, \chi_{[0,1/2]}$ are
all bounded and the inclusion map $\iota\dvtx\ell^2(\mathbb C) \to
\ell
^4(\mathbb C)$ is also a bounded operator, we have \mbox{$\|\Pi\|_{2 \to4}
<\infty$}.

%Clearly, $\|\Pi_k\|_{2 \to4} $ are nondecreasing and
%$\|\Pi_k\|_{2 \to4} \le\|\Pi\|_{2 \to4} $ for all $k \ge1$.
%
%Let $\eps>0$, and let $v$ be a unit vector so that $\|\Pi
%v\|_4>\|\Pi\|_{2\to4}-\eps$. Let $v_k$ be unit vectors supported
%on $\{1,\ldots,k\}$ so that $\|v_k-v\|_2\to0$. Then $$|\|\Pi
%v_k\|_4-\|\Pi v\|_4|\le\|\Pi(v_k-v)\|_4\to0$$ and
%$\|\Pi(v_k)\|_4=\|\Pi_k(v_k)\|_4$, so $\liminf\|\Pi_k\|_{2\to4}
%

It will be convenient to think of $\Pi_{2k+1}$ as a linear operator
acting on the space $\ell^2(\mathbb C)$ with the representation $\Pi
_{2k+1}(i, j) = \Pi(i, j)$ for $|i|,|j|\le k$ and $0$ otherwise.
Clearly, $\|\Pi_k\|_{2\to4}$ is increasing and bounded above by $\|
\Pi
\|_{2\to4}$.

For the other direction, consider a sequence of unit vectors $\mathbf
{v}_m\in
\ell^2(\mathbb C)$ supported on $[-m,m]$ so that $\|\Pi\mathbf{v}_m\|
_4\to\|
\Pi\|_{2\to4}$. Then for $k\ge m$ we have
\[
\|\Pi\mathbf{v}_m\|_4^4 - \|\Pi_{2k+1}
\mathbf{v}_m \|^4_{4} = \sum
_{i: |i|>k} \bigl|(\Pi\mathbf{v}_m) (i)\bigr|^4 = \sum
_{i: |i|>k} \biggl\llvert\sum
_{ j: |j| \le m} \Pi(i, j)\mathbf{v}_m(j) \biggr\rrvert
^4.
\]
Since $|\Pi(i, j)| \le|i-j|^{-1}$ and $|\mathbf{v}_m(j)| \le1$, the inside
sum is bounded above by $(2m+1)(|i|-m)^{-1}$. This gives the upper bound
\[
(2m+1)^4\sum_{\ell: |\ell| > k-m} \frac{1}{\ell^4}
\to0 \qquad \mbox{as } k\to\infty.
\]
Letting $k\to\infty$ and then $m\to\infty$ completes the proof.
\end{pf}\eject

Given $\tau>0$, by \eqref{eqPiknorm} and Lemma~\ref
{Lemoperatornormconv}, we can find $k \ge1$ sufficiently large and
a vector $\mathbf{u}= (u_1, u_2, \ldots, u_k) \in\mathbb R^k$ with
$\| \mathbf{u}\|
_2< 1$ such that\break $ \lambda_1(\Pi_k \operatorname{diag}(\mathbf{u})
\Pi_k) > \|\Pi\|
^2_{2 \to4} -\tau/2$. By perturbing the coordinates a little, if
necessary, we can also assume $u_s \ne0$ for each $ 1\le s \le k$. Now
choose $\eta>0$ sufficiently small such that:
\begin{longlist}[(1)]
\item[(1)] $0 \notin(u_s -\eta, u_s+ \eta)$ for each $ 1 \le s \le k$.
\item[(2)] If $\mathbf{v}= (v_1, v_2, \ldots, v_k) \in[u_1 -\eta,
u_1+ \eta]
\times\cdots\times[u_k - \eta, u_k + \eta]$, then $ \lambda_1(\Pi_k
\operatorname{diag}(\mathbf{v}) \Pi_k) > \|\Pi\|^2_{2 \to4} -\tau$.
\item[(3)] $\sup\{ \| \mathbf{v}\|_2 \dvtx\mathbf{v}\in[u_1 -\eta
, u_1+ \eta] \times
\cdots\times[u_k - \eta, u_k + \eta] \} < 1$.
\end{longlist}
Finally choose $\eps>0$ small such that $[-\eps, \eps] \cap(u_s
-\eta,
u_s+ \eta) = \varnothing$ for each $ 1 \le s \le k$.

By Proposition~\ref{propsecondmoment} we know that one of the events
$A_i, 1 \le i \le N $ happens with high probability.
If $A_i$ occurs, then the points $ip-\lceil\log n \rceil^3, \ldots, ip
+ k+\lceil\log n \rceil^3$ are contained in a single partition block
$J\in\Lambda$. This is because when $A_i$ occurs, $d_{ip +j} \in(u_s
- \eta, u_s +\eta)$ and hence $|d_{ip+j}| > \eps$ for each
$j=1,\ldots
,k$ and because of the property of our random partition which
guarantees that each (random) partition block always has a padding of
two invisible bricks (of length between $\lceil\log n \rceil^3$ and
$4 \lceil\log n \rceil^3$) from each side. On the other hand, since
two consecutive invisible bricks cannot belong to the same partitioning
block, $J$ has no other point from $S$ except the $k$ points $ip +1,
ip+2, \ldots, ip+k$. Write $F: = \{ip +1, ip+2, \ldots, ip+k\}$.
Therefore, if $A_i$ happens, then
\begin{eqnarray*}
\hspace*{-4pt}&&\frac{\lambda_1( \bP[J ] \bD^{\eps} [J] \bP[J ])} {\sqrt{2\log
n} } \\
\hspace*{-4pt}&&\qquad\ge\frac{\lambda_1( \bP[F] \bD^{\eps} [F] \bP[F ])} {\sqrt
{2\log n}
}
\\
\hspace*{-4pt}&&\qquad \ge\inf\bigl\{ \lambda_1 \bigl( \bP[F] \operatorname
{diag}(\mathbf{v}) \bP[F ]
\bigr)\dvtx\mathbf{v}\in[u_1 -\eta, u_1+ \eta] \times\cdots
\times[u_k - \eta, u_k + \eta] \bigr\}.
\end{eqnarray*}
By the convergence \eqref{eqrateofconv} of $\bP$ to $\Pi$ this equals
\begin{eqnarray*}
&&\inf\bigl\{ \lambda_1 \bigl( \Pi_k \operatorname{diag}(\mathbf{v})
\Pi_k \bigr)\dvtx\mathbf{v}\in[u_1 -\eta, u_1+ \eta]
\times\cdots\times[u_k - \eta, u_k + \eta] \bigr\} - O
\bigl(n^{-1}\bigr)
\\
&&\qquad \ge\| \Pi_k \|^2_{2 \to4} - \tau- O
\bigl(n^{-1}\bigr).
\end{eqnarray*}
So, by Lemma~\ref{Lemoperatornormconv} we get with high probability,
\[
\max_{ J\in\Lambda: J\cap S \ne\varnothing} \frac{\lambda_1( \bP
[J ]
\bD^{\eps} [J] \bP[J ])} {\sqrt{2\log n} } \ge\| \Pi_k
\|^2_{2 \to4} - \tau- O\bigl(n^{-1}\bigr)= \| \Pi
\|^2_{2 \to4} - \tau-o (1).
\]

This yields the claim of Proposition~\ref{prolowerbound}.
%
%Lemma~\ref{lemtightness}, it suffices to show that

%sA #&#
\begin{appendix}\label{app}
\section*{Appendix}

%sA.1 #&#
\subsection{Invariance principle}

Let $\bX= (X_1, \ldots, X_r)$ and $\bY= (Y_1, \ldots, Y_r)$ be two
vectors of independent random variables with finite second moments,
taking values in some open interval $I$\vadjust{\goodbreak} and satisfying, for each $i$,
$\E[X_i] = \E[Y_i]$ and $\E[X_i^2] = \E[Y_i^2]$. We shall also assume
that $\bX$ and $\bY$ are defined on the same probability space and are
independent. The following lemma is an immediate generalization of
Theorem~1.1 of \citet{chatterjee05} (based on Lindeberg's approach to
the CLT). We need this more detailed version because we will use the
invariance principle under the moderate deviation regime.\looseness=-1
%
%le17 #&#
\begin{lemma} \label{leminvariance}
Let $f = (f_1, f_2, \ldots, f_m)\dvtx I^r \to\mathbb R^m$ be thrice
continuously differentiable. If we set $\bU= f(\bX)$ and
$\bV=f(\bY)$, then for any thrice continuously differentiable $g\dvtx
\mathbb R^m \to\mathbb R$,
\[
\bigl| \E\bigl[g(\bU)\bigr] - \E\bigl[g(\bV)\bigr] \bigr| \le\sum
_{i=1}^r \E[ R_i ] + \sum
_{i=1}^r \E[T_i],
\]
where
\begin{eqnarray*}
R_i &:=& \frac{1}{6} |X_i|^3
\times\sup_{ x \in[\min(0, X_i), \max(0,
X_i)] } \bigl| h_i( X_1, \ldots,
X_{i-1}, x, Y_{i+1}, \ldots, Y_r)\bigr|,
\\
T_i &:=& \frac{1}{6} |Y_i|^3
\times\sup_{ y \in[\min(0, Y_i), \max(0,
Y_i)] } \bigl|h_i( X_1, \ldots,
X_{i-1}, y, Y_{i+1}, \ldots, Y_r)\bigr|,
\\
h_i(\mathbf{x}) &:= &\sum_{ \ell, p, q = 1}^m
\partial_\ell\,\partial_p \,\partial_q g\bigl(f(
\mathbf{x})\bigr) \partial_i f_\ell(\mathbf{x}) \,\partial_i
f_p(\mathbf{x}) \,\partial_i f_q(\mathbf{x})
\\
&&{}+ 3 \sum_{ \ell, p= 1}^m \partial_\ell\,
\partial_p g\bigl(f(\mathbf{x})\bigr) \,\partial^2_i
f_\ell(\mathbf{x}) \,\partial_i f_p(\mathbf{x}) + \sum
_{ \ell= 1}^m \,\partial_\ell g
\bigl(f(\mathbf{x})\bigr) \,\partial^3_i f_\ell(\mathbf{x}).
\end{eqnarray*}
\end{lemma}

\begin{pf}
The lemma can be proved by merely imitating the steps of \citeauthor{chatterjee05} [(\citeyear{chatterjee05}), Theorem~1.1],
but we include a proof here for sake of completeness.
Let $H\dvtx I^r \to\mathbb R$ be the function $H(\mathbf{x}):= g (
f(\mathbf{x}))$. It
is a routine computation to verify that $ \partial_i^3 H(\mathbf{x})
= h_i(\mathbf{x}
)$ for all $i$.
For $0 \le i \le r$, define the random vectors $\bZ_i:= (X_1, X_2,
\ldots, X_{i-1}, X_i, Y_{i+1}, \ldots, Y_n)$ and $\bW_i:= (X_1, X_2,
\ldots, X_{i-1}, 0, Y_{i+1}, \ldots,  Y_n)$ with obvious meanings for $i
= 0$ and $i =r$. For $1 \le i \le r$, define
\[
\err^{(1)}_i= H(\bZ_i) - X_i\,
\partial_i H(\bW_i) - \tfrac12 X_i^2\,
\partial_i H(\bW_i)
\]
and
\[
\err^{(2)}_i= H(\bZ_{i-1}) - Y_i\,
\partial_i H(\bW_i) - \tfrac12 Y_i^2\,
\partial_i H(\bW_i).
\]
Hence, by Taylor's remainder theorem and the above observation about
the third partial derivatives of $H$, it follows that
\[
\bigl|\err^{(1)}_i\bigr| \le R_i\quad \mbox{and} \quad\bigl|
\err^{(2)}_i\bigr| \le T_i.
\]
For each $i$, $X_i$, $Y_i$ and $\bW_i$ are independent. Hence,
\[
\E\bigl[ X_i\, \partial_i H(\bW_i)\bigr]
- \E\bigl[Y_i\, \partial_i H(\bW_i)\bigr]
= \E[ X_i - Y_i ] \cdot\E\bigl[H(\bW_i)
\bigr] = 0.\vadjust{\goodbreak}
\]
Similarly, $\E[X_i^2 \,\partial_i H(\bW_i)] = \E[Y_i^2\, \partial_i
H(\bW
_i)]$. Combining all these ingredients, we obtain
\begin{eqnarray*}
\bigl|\E\bigl[g(\bU)\bigr] - \E\bigl[g(\bV)\bigr] \bigr| &=& \Biggl| \sum
_{i=1}^r \bigl( \E\bigl[H(\bZ_i)
\bigr] - \E\bigl[H(\bZ_{i-1})\bigr] \bigr) \Biggr|
\\
&\le&\Biggl| \sum_{i=1}^r \E\bigl[
X_i \,\partial_i H(\bW_i)\bigr] + \frac12
X_i^2 \,\partial_i H(\bW_i)+
\err^{(1)}_i
\\
&&{}- \sum_{i=1}^r \E\bigl[
Y_i \,\partial_i H(\bW_i)\bigr] + \frac12
Y_i^2 \,\partial_i H(\bW_i)+
\err^{(2)}_i \Biggr |
\\
& \le&\sum_{i=1}^r \E[ R_i
] + \sum_{i=1}^r \E[T_i],
\end{eqnarray*}
which completes the proof.
\end{pf}

%sA.2 #&#
\subsection{Covariances between the eigenvalues of random circulant}

%le18 #&#
\begin{lemma} \label{lemdjmeancov}
Let $a_0, a_1, \ldots, a_n$ be independent mean zero, variance one
random variables. Define
\[
d_j = \frac{1}{\sqrt{2n}} \Biggl[ \sqrt{2} a_0 +
(-1)^j\sqrt{2} a_n + 2\sum_{k=1}^{n-1}
a_k \cos\biggl( \frac{2 \pi jk}{2n} \biggr) \Biggr],\qquad 0 \le j < 2n.
\]
Then $d_j = d_{2n-j}$ for $0< j< 2n$. Moreover, the random variables
$d_j, 0 \le j \le n$ have mean $0$, and their covariances are given by
\[
\E[d_{j} d_{k} ] = \cases{ %
2, &\quad
$\mbox{if } j =k \in\{0, n\},$
\vspace*{2pt}\cr
1,& \quad$\mbox{if } 0< j=k< n,$
\vspace*{2pt}\cr
0, &\quad $\mbox{if } 0 \le j \ne k \le n.$}
\]
\end{lemma}
\begin{pf}
The fact that $d_{j} = d_{2n-j}$ for $0 < j< 2n$ and zero mean property
is immediate from the definition of $d_j$. Since $a_0, a_1, \ldots,
a_n$ be independent with variance one,
\begin{eqnarray*}
\E[d_j d_k ] &=& \frac{1}{2n} \Biggl[ 2+
2(-1)^{j+k} + 4\sum_{\ell=1}^{n-1}
\cos\biggl( \frac{2 \pi j\ell}{2n} \biggr) \cos\biggl( \frac{2
\pi k\ell}{2n} \biggr)
\Biggr]
\\
&=& \frac{1}{2n} \Biggl[ 2+ 2(-1)^{j+k} + 2\sum
_{\ell=1}^{n-1} \cos\biggl( \frac{2 \pi(j-k)\ell}{2n} \biggr) \\
&&\hspace*{82pt}{}+
2\sum_{\ell=1}^{n-1}\cos\biggl(
\frac{2 \pi(j+k)\ell}{2n} \biggr) \Biggr].
\end{eqnarray*}
Plugging in $x = \frac{2 \pi m}{2n}, m = 0, 1, 2, \ldots, 2n$ into the
well-known Dirichlet kernel formula,
\[
1+ 2\sum_{ \ell=1}^{n-1} \cos(\ell x) =
\frac{ \sin((n-1/2)x) }{\sin(x/2)},
\]
we obtain
\[
\sum_{ \ell=1}^{n-1} \cos\biggl(
\frac{2 \pi m \ell}{2n} \biggr) = -\frac
{1+ (-1)^{m}}{2},
\]
unless $m=0$ or $2n$ when the sum equals to $(n-1)$. Using the above
formula, the covariances $\E[d_j d_k] $ can be easily computed.
\end{pf}

%sA.3 #&#
\subsection{Optimization problem and connection to the sine kernel}
For any $\delta>0$, let $\upsilon_{\delta}$ be the indicator function
$\1_{[-\delta/2, \delta/2]}$. For a complex valued function $f$ defined
on $\mathbb R$, we define the involution $f^*$ by $f^*(x) = \overline
{f(-x)}$. Given two functions $f$ and $g$ on $\mathbb R$ their
convolution $f \star g$ is defined as $(f \star g)(x) = \int_{\mathbb
R} f(x-y) g(y) \,dy$ provided the integral makes sense. Let $f \star_\T
g(x) = \int_{-1/2}^{1/2} \tilde f(x-y) \tilde g(y) \,dy $ denote the
convolution of the two functions $f$ and $g$ are in $L^2(\T)$ where
$\tilde f$ and $\tilde g$ are the periodic extension of $f$ and $g$,
respectively, on the whole real line. Let
$\hat f(t) = \int_{\mathbb R} e^{ -2 \pi\ci x t} f(x)\,dx$ be the usual
Fourier transform of $f \in L^2(\mathbb R)$. Let $\psi^{-1}$ be the
discrete Fourier transform of from $L^2(\T)$ to $\ell^2(\mathbb C)$.
Below we collect some basic facts of the usual and discrete Fourier
transform which we will need later:
\begin{longlist}[(1)]
\item[(1)] $\psi^{-1}\dvtx L^2(\T) \to\ell^2(\mathbb C)$ and $\hat
{} \dvtx L^2(\mathbb R) \to L^2(\mathbb R)$ are isometries.
\item[(2)] $\overline{\psi^{-1}(f)} = \psi^{-1}(f^*)$ and
\hspace*{3pt}$\overline
{\hspace*{-3pt}\widehat f} = \widehat{f^*}$.
\item[(3)] $\psi^{-1}( f \star_\T g)(k) = \psi^{-1}(f)(k)\psi
^{-1}(g)(k)$ for all $k \in\mathbb Z$ and for all $f, g \in L^2(\T)$.
When $f, g \in L^2(\T)$, then $\widehat{f \star g} = \hat f \cdot
\hat g$.
\item[(4)] If $f$ and $g$ are supported on $[0, 1/2]$, then $f\star
_\T
g = f\star g$.
\item[(5)] $ \hat\ups_1 (t) = \frac{\sin(\pi t )}{\pi t}$.
\end{longlist}
Note that $ \Sin(f)(x) = \hat\ups_1 \star f(x) \mbox{ for } f \in
L^2(\mathbb R)$.
The next lemma establishes the connection between the $2 \to4$ norm of
the operator $\Pi$ (defined in Section~\ref{subsecpropP}) and the
$2 \to4$ norm of the integral operator $\Sin$.
%
%le19 #&#
\begin{lemma} \label{lemsinepiconn}
The following holds true:
\[
\| \Pi\|_{2 \to4}^2 = \frac{1}{\sqrt2} \| \mbox{\em$\Sin$}
\|_{2 \to4}^2.
\]
\end{lemma}
Let us define for $\delta>0$,
%
%
%eA.1 #&#
\begin{equation}
\label{optprob} K_\delta:= \sup\bigl\{ \bigl\| (f
\ups_{\delta})^* \star(f \ups_{\delta}) \bigr\|_2\dvtx f \in
L^{2}(\mathbb R), \| f \|_2 \le1 \bigr\}.
\end{equation}
%
%le20 #&#
\begin{lemma} \label{lscalingKdelta}
Let $K_\delta$ be as above. Then $K_\delta= \delta^{1/2} K_1$.
\end{lemma}
\begin{pf}%{Proof of Lemma~\ref{lscalingKdelta}}
For any $f \in
L^2(\mathbb R)$,
\[
\bigl\| (f \ups_{\delta})^* \star(f \ups_{\delta})
\bigr\|^2_2 = \int\biggl\llvert\int f \ups_\delta(x+t)
\overline{f \ups_\delta(t)} \,dt \biggr\rrvert^2 \,dx.
\]
After a change of variables $ s = t/\delta$ and $y = x /\delta$ the
above integral is same as
%
%
%eA.2 #&#
\begin{equation}
\label{scalingcrossCorr} \delta^{3} \int\biggl\llvert
\int f \ups_\delta\bigl(\delta(s+ y) \bigr) \overline{f
\ups_\delta(\delta s)} \,ds \biggr\rrvert^2 \,dy.
\end{equation}
Keeping in mind that $\ups_\delta( \delta s) = \ups_1(s)$ and replacing
$f$ by $g(x):= \delta^{1/2} f(\delta x)$ in~\eqref{scalingcrossCorr},
we obtain
\[
\eqref{scalingcrossCorr} = \delta\int\biggl\llvert\int g
\ups_1 \bigl((s+ y) \bigr) \overline{g \ups_1( s)} \,ds
\biggr\rrvert^2 \,dy.
\]
Since $\|g\|_2 = \|f\|_2$, it follows that $K_\delta^2 = \delta K_1^2$
which completes the proof of the lemma.
\end{pf}

\begin{pf*}{Proof of Lemma~\ref{lemsinepiconn}}
The proof consists of a series of elementary steps. Let~$\kappa$ be the
indicator function $\1_{[0, 1/2]}$. Applying
the Fourier transform from $\ell^2(\mathbb{Z})$ to $L^2(\mathbb
T)$, we get
\begin{eqnarray*}
\| \Pi\|_{2 \to4}^2 &=& \sup\bigl\{ \| \overline{\Pi\mathbf{v}}
\odot\Pi\mathbf{v}\|_2\dvtx\mathbf{v}\in\ell^2(\mathbb Z), \|
\mathbf{v}
\|_2 \le1 \bigr\}
\\
&=& \sup\bigl\{ \bigl\| \overline{\psi^{-1}(f\cdot\kappa)} \odot\psi
^{-1}(f\cdot\kappa) \bigr\|_2\dvtx f \in L^{2}(\T), \| f
\|_2 \le1 \bigr\}.
\end{eqnarray*}
The properties of the Fourier transform further imply
\[
\overline{\psi^{-1}(f\cdot\kappa)} \odot\psi^{-1}(f\cdot
\kappa) = \psi^{-1}\bigl( (f\cdot\kappa)^*\bigr) \odot
\psi^{-1}(f\cdot\kappa)= \psi^{-1}\bigl( (f\cdot\kappa)^*
\star_\T(f\cdot\kappa)\bigr),
\]
and since we convolve functions supported on $[0,1/2]$, we might as
well do the entire optimization on the real line to get
\[
\| \Pi\|_{2 \to4}^2 = \sup\bigl\{\bigl \| (f\cdot\kappa)^*
\star(f\cdot\kappa) \bigr\|_2\dvtx f \in L^{2}(\mathbb R), \| f
\|_2 \le1 \bigr\}
.
\]
By translating the function $ f \in L^{2}(\mathbb R)$ via the map $f
\mapsto f(\cdot+ 1/4)$ in the above optimization problem, we see that
$\|\Pi\|_{2 \to4}^2 =K_{1/2}$. This equals $K_1/\sqrt{2}$ by
Lemma~\ref{lscalingKdelta}. Now note that
\[
\bigl\| (f \ups_{1})^* \star(f \ups_{1}) \bigr\|_2=\|
\overline{\widehat{f\ups_{1} }} \cdot\widehat{f
\ups_{1}} \|_2=\| \widehat{f \ups_{1}} \|
_4^2=\bigl\|\Sin(\hat f) \bigr\|_4^2
\]
and so
\[
K_1=\sup\bigl\{ \bigl\| \Sin(\hat f) \bigr\|_4^2\dvtx\hat f \in L^{2}(\mathbb
R), \| \hat f \|_2 \le1 \bigr\} =
\| \Sin\|_{2 \to4}^2.
\]
This completes the proof of the lemma.
\end{pf*}
As far as we know, the explicit value of constant $K_1$ of \eqref
{optprob} is not known in the literature. However, the maximization
problem given in~\eqref{optprob} has been studied in \citet{Garsia68}.
We list below some of the interesting results from \citet{Garsia68}:
\begin{itemize}
\item$K_1 = \sup\{ ( \int_{\mathbb R} (\int_{\mathbb R}
f(x+t) f(t) \,dt )^2 \,dx )^{1/2}\dvtx f \in\mathcal F \}$ where
$\mathcal F$ is the class of all real-valued functions $f$ satisfying
$f(x) \ge0, f(x) = f(-x)$ for all $x \in\mathbb R, f(x) \ge f(y)$
for $0 \le x \le y$ and $f(x) = 0$ for $|x| \ge1/2$ and $\int_{
-1/2}^{1/2} f^2(x) \,dx = 1$.
\item There exists a unique $f \in\mathcal F$ such that $ ( \int
_{\mathbb R} (\int_{\mathbb R} f(x+t) f(t) \,dt )^2 \,dx
)^{1/2}= K_1$.
\item$K_1^2 = 0.686981293033114600949413 \ldots!$
\end{itemize}
\end{appendix}

\section*{Acknowledgments}
A. Sen thanks Manjunath Krishnapur for pointing out the usefulness
of the invariance principle in the world of random matrices. Most
of the research has been conducted while A. Sen visited University
of Toronto and Technical University of Budapest in June 2010 and
July 2011, respectively. B.~Vir\'{a}g thanks M\'at\'e Matolcsi for many
interesting discussions about the final optimization problem. We
also thank him and Mihalis Kolountzakis for the reference
\citet{Garsia68}.
%
% imsref loaded by akundreckaite, 2013-09-02 10:26:21
%

\printaddresses

\end{document}